\documentclass[11pt,twoside,reqno]{amsart}
 \usepackage[top=1in, bottom=1in, left=1in, right=1in]{geometry}
\usepackage{graphicx}
\usepackage{amsmath,latexsym,amssymb}
\usepackage{amsfonts}
\usepackage{mathtools,esint}
\usepackage{amsmath}
\usepackage{amscd,amsthm}
\usepackage{cases}
\usepackage{stmaryrd}
\usepackage{fancyhdr}
\usepackage{commath,enumerate}
\usepackage{indentfirst, hyperref}
\usepackage{caption, subcaption}
\usepackage{dsfont}
\usepackage{pgf,tikz}
\usepackage{xcolor}
\usepackage{mathrsfs}
\usetikzlibrary{arrows,calc,positioning}
\usepackage{float}
\numberwithin{figure}{section}
\usepackage[bottom]{footmisc}
\usepackage{datetime}

\usepackage[numbers,sort&compress]{natbib}
\usepackage{cancel}

\usepackage{mathtools}

\numberwithin{equation}{section}
\newtheorem{theorem}{Theorem}[section]

\theoremstyle{definition}
\newtheorem{definition}[theorem]{Definition}
\newtheorem{remark}[theorem]{Remark}

\newtheorem{lem}{Lemma}

\newcommand{\dps}{\displaystyle}

\newcommand{\fr}{\frac}
\newcommand{\pa}{\partial}
\newcommand{\les}{\lesssim}

\newcommand{\Rm}[1]{
  \textup{\uppercase\expandafter{\romannumeral#1}}
}

\newcommand{\ub}{\mathbf{u}}
\newcommand{\nb}{\mathbf{n}}

\def\ds{\displaystyle}

\newcommand{\be}{\begin{enumerate}}
\newcommand{\ee}{\end{enumerate}}
\newcommand{\bi}{\begin{itemize}}
\newcommand{\ei}{\end{itemize}}
\newcommand{\ba}{\begin{array}}
\newcommand{\ea}{\end{array}}

\makeatletter
\newcommand{\bal}{\@ifstar{\@bals}{\@bal}}
\def\@bals#1\eal{\begin{align*}#1\end{align*}}
\def\@bal#1\eal{\begin{align}#1\end{align}}
\makeatletter

\allowdisplaybreaks[4]

\let\oldtocsection=\tocsection

\let\oldtocsubsection=\tocsubsection

\let\oldtocsubsubsection=\tocsubsubsection

\renewcommand{\tocsection}[2]{\hspace{0em}\oldtocsection{#1}{#2}}
\renewcommand{\tocsubsection}[2]{\hspace{1em}\oldtocsubsection{#1}{#2}}
\renewcommand{\tocsubsubsection}[2]{\hspace{2em}\oldtocsubsubsection{#1}{#2}}

\usepackage{etoolbox}
\makeatletter
\patchcmd{\@settitle}{\uppercasenonmath\@title}{}{}{}
\patchcmd{\@setauthors}{\MakeUppercase}{}{}{}
\makeatother

\begin{document}
\title[Initial-boundary value problems for Poiseuille flow of Ericksen-Leslie model]{\bf\Large  Initial-boundary value problems for Poiseuille flow of nematic liquid crystal via full Ericksen-Leslie model}

   \author[Geng Chen]{Geng Chen$^{(*)}$ \\
(*) Department of Mathematics, University of Kansas,
Lawrence, KS, U.S.A.
\\
e-mail:~gengchen@ku.edu
}
\author[Yanbo Hu]{\\Yanbo Hu$^{(**)}$\\(**) Department of Mathematics, Hangzhou Normal University, Hangzhou, China\\ e-mail:~yanbo.hu@hotmail.com
}
\author[Qingtian Zhang]{\\Qingtian Zhang$^{(***)}$ \\
(***) Department of Mathematics, West Virginia University,
Morgantown, WV 26501, U.S.A.
\\
e-mail:~qingtian.zhang@mail.wvu.edu
}
\date{}

\begin{abstract}
In this paper, we study the initial-boundary value problem for the Poiseuille flow of hyperbolic-parabolic Ericksen-Leslie model of nematic liquid crystals in one space dimension. Due to the quasilinearity, the solution of this model in general forms cusp singularity. We prove the global existence of H\"older continuous solution, which may include cusp singularity, for initial-boundary value problems with different types of boundary conditions.


\end{abstract}


\maketitle


\section{Introduction}
The hydrodynamic theory of incompressible liquid crystals was established by Ericksen \cite{Eric61, Eric87, Eric90} and Leslie \cite{Les79}.
The Ericksen-Leslie's system is written as
\begin{equation}\label{ELsys}
\left\{
\begin{array}{ll}
\rho\dot \ub+\nabla P=\nabla \cdot\sigma-\nabla(\frac{\partial W}{\partial \nabla \nb}\otimes\nabla \nb),\\
\nabla \cdot\ub=0,\\
\nu\ddot \nb=\lambda\nb-\frac{\partial W}{\partial \nb}-\mathbf g+\nabla\cdot(\frac{\partial W}{\partial\nabla\nb}),\\
|\nb|=1,
\end{array}
\right.
\end{equation}
where $\ub$ is the velocity, $\nb$ is the director field of the liquid crystal molecules, $\rho$ is the constant density, $P$ is the pressure, $\nu$ is the inertia coefficient of the director $\nb$. $\mathbf g$ and $\sigma$ are the kinematic transport and the viscous stress tensor, respectively, which satisfy
\begin{align*}
&\mathbf g=\gamma_1N+\gamma_2 D\nb,\quad N=\dot\nb-\omega \nb,\\
&D=\frac12(\nabla\ub+\nabla^t\ub),\quad \omega=\frac12(\nabla\ub-\nabla^t\ub), \\
&\sigma=\alpha_1(\nb^tD\nb)\nb\otimes\nb+\alpha_2 N\otimes\nb+\alpha_3\nb\otimes N+\alpha_4D+\alpha_5(D\nb)\otimes\nb+\alpha_6\nb\otimes(D\nb),
\end{align*}
where $\alpha_1, \cdots, \alpha_6, \gamma_1, \gamma_2$ are physical coefficients satisfying (see \cite{Les79,CHL20})
$$\gamma_1=\alpha_3-\alpha_2, \gamma_2=\alpha_6-\alpha_5, \alpha_2+\alpha_3=\alpha_6-\alpha_5,$$
$$\alpha_4>0, 2\alpha_1+3\alpha_4+2\alpha_5+2\alpha_6>0, \gamma_1>0, $$
$$2\alpha_4+\alpha_5+\alpha_6>0, 4\gamma_1(2\alpha_4+\alpha_5+\alpha_6)>(\alpha_2+\alpha_3+\gamma_2)^2.$$

If the orientation order parameters of nematic materials are treated as a unit vector ${\nb}\in \mathbb S^2$, the director, then the Oseen-Frank energy density determines the macrostructure of the crystal structure (\cite{Les79})
 \begin{align*}\begin{split}
2W(\nb,\nabla \nb)=&K_1(\nabla \cdot \nb)^2+K_2(\nb\cdot(\nabla\times \nb))^2+K_3|\nb\times (\nabla \times \nb)|^2\\
&+(K_2+K_4)[\hbox{tr}(\nabla \nb)^2-(\nabla \cdot \nb)^2 ],
\end{split}
\end{align*}
where $K_j$, $j=1,2,3$, are the positive constants representing splay, twist, and bend effects respectively, with
$K_2\geq |K_4|$, $2K_1\geq K_2+K_4$.

Successful theories with a wide range of interesting properties  have been established for the equilibrium theory (the elliptic case on $\nb$) and the evolutionary theory (the parabolic case on $\nb$) when the inertial effect in \eqref{ELsys} is neglected, i.e. when $\nu=0$. See a partial list of references in  \cite{HardtK87, linlius01, linwangs14, lin89, huanglinwang14, linwang16, hongxin12, LTX16, wangwang14,WZZ13}.

However, there are very few studies on the full Ericksen-Leslie system, including the inertial effect, i.e. when $\nu=0$. Many fundamental problems including global wellposedness are still wide open. When one considers a special Oseen-Frank potential $W=|\nabla \nb|^2$, which makes the wave equations on $\nb$ essentially semilinear, for small data problem, Jiang and his collaborators obtained a series of existence results on regular solutions \cite{JL19, JLT19, CJL21, HJLZ21, JL22}.

In general, the solution of \eqref{ELsys} might form a finite time singularity even in one space dimension, due to the quasilinearity in the wave equation of $\nb$. See \cite{CHL20} for the formation of finite time cusp singularity for the Poiseuille flow, and \cite{H2} for the formation of cusp singularity in multiple space dimension. For solutions in multiple space dimension, Chen-Huang-Xu in \cite{CHX} found another type of singularity for \eqref{ELsys} due to the geometric effect, similar as the one for the semilinear wave map equation.

In \cite{CHL20}, the first large data global existence result on the Cauchy problem of \eqref{ELsys} was established by Chen-Huang-Liu in \cite{CHL20} for the 1-d Poiseuille flow, where the solution may include cusp singularity.

In this paper, we establish some global existence results for the initial-boundary value problem of the Poiseuille flow for nematic liquid crystals via the full Ericksen-Leslie model \eqref{ELsys}. More precisely, for the 1-d Poiseuille flow with $\ub=(0,0, u)^t, \nb=(\sin\theta, 0, \cos\theta)^t$, the system of $u$ and $\theta$ becomes
\begin{equation}\label{sys}
\left\{
\begin{array}{ll}
\rho u_t=(g(\theta)u_x+h(\theta)\theta_t)_x,\\
\nu\theta_{tt}+\gamma_1\theta_t=c(\theta)(c(\theta)\theta_x)_x-h(\theta)u_x,
\end{array}\right.
\end{equation}
where
\begin{equation*}
\begin{aligned}
&g(\theta):=\alpha_1\sin^2\theta\cos^2\theta+\frac{\alpha_5-\alpha_2}2\sin^2\theta+\frac{\alpha_3+\alpha_6}2\cos^2\theta+\frac{\alpha_4}2,\\
&c(\theta):=\sqrt{K_1\cos^2\theta+K_3\sin^2\theta},\\
&h(\theta):=\alpha_3\cos^2\theta-\alpha_2\sin^2\theta=\frac{\gamma_1+\gamma_2\cos(2\theta)}2.
\end{aligned}
\end{equation*}
See the derivation of this model in \cite{CHL20}.
Notice that $g(\theta), c(\theta), h(\theta)$ are smooth with respect to $\theta$, and $g(\theta), c(\theta), h(\theta), g'(\theta), c'(\theta), c''(\theta), h'(\theta)$ are uniformly bounded.

There are several physically interesting boundary conditions. For the velocity $u$, we have the following possible choices. We denote an important quantity
\[
J=g(\theta)u_x+h(\theta)\theta_t.
\]
\begin{itemize}
\item Nonslip boundary condition. If the boundary is a solid wall, we can propose the nonslip boundary condition $u=0$ on the boundary. If the boundary is moving, we can also propose $u=f(t)$ on the boundary with $f(t)$ is a given function with $t$. This corresponds to the Dirichlet boundary condition on $u$.
\item Stress-free boundary condition. On the boundary the shear stress is zero. $(1,0,0)D=(0,0,J)^t={\mathbf 0}$, where $(1,0,0)$ is the normal direction of the boundary. This boundary condition can be seen in the jet with a free boundary. Under the assumption of the Poiseuille flow, there is no in-flow or out-flow, so the boundary is fixed. This corresponds to the Neumann boundary condition on $u$.
\item Navier boundary condition. On the boundary, the shear stress is proportional to the tangential velocity, that is $J=-\gamma u$. This corresponds to the Robin boundary condition on $u$.
\end{itemize}

For the director $\nb$, one can propose the following boundary conditions \cite{Ste}.
\begin{itemize}
\item Strong anchoring condition. On the boundary, $\nb$ is given. This corresponds to the Dirichlet boundary condition.
\item No anchoring condition. $\frac{\partial W}{\partial n_{i,j}}\nu_j=0$.
\item Weak anchoring condition. $\frac{\partial W_F}{\partial n_{i,j}}\nu_j+\frac{\partial W_s}{\partial n_{i}}=\gamma n_i$.
\end{itemize}

In this paper, we will consider all combinations of boundary conditions, excluding the Navier boundary condition.

\subsection{Main result}
For simplicity, we only
consider a special case when $\alpha_1=0$, $\alpha_5-\alpha_2=\alpha_3+\alpha_6$, so that $g(\theta)$ is a constant $\frac{\alpha_3+\alpha_4+\alpha_6}2$. Without loss of generality, we assume $g(\theta)=1$, $\nu=1, \gamma_1=2, \gamma_2=0$ and $\rho=1$. So $h(\theta)=1$. Then the system is written as
\begin{equation}\label{1.1}
\left\{
\begin{array}{ll}
 u_t=(u_x+\theta_t)_x,\\
\theta_{tt}+2\theta_t=c(\theta)(c(\theta)\theta_x)_x-u_x,
\end{array}\right.
\end{equation}
where the function $c(\cdot)$ is a $C^2$ function satisfying
\begin{align}\label{1.6}
0<C_L\leq c(\cdot)\leq C_U<\infty,\  \ |c'(\cdot)|\leq C_1<\infty,
\end{align}
for some positive constants $C_L$, $C_U$ and $C_1$. In this case, $J=u_x+\theta_t$.

The first equation of \eqref{1.1} has constant coefficients. This gives us some technical advantage than \eqref{sys}. In \cite{CHL20}, the authors first considered the simplified system \eqref{1.1}. Later in \cite{CLS}, the global existence result of Cauchy problem was extended to the general system \eqref{sys}. We conject the result in this paper on \eqref{1.1} (Theorem \ref{thm_0}) still holds for \eqref{sys}.

The following global existence theorem on the initial boundary value problem of \eqref{1.1} on $x\in[0,\pi]$ is our main theorem. For the convenience on notations, we choose the domain as $x\in[0,\pi]$. The result holds for any bounded domain  $x\in[a,b]$ by the same proof.
\begin{theorem}\label{thm_0}
Assume initially
 \begin{equation}\label{1.2}
 u(x,0)=u_0(x)\in H^1([0,\pi]),~~ \theta(x,0)=\theta_0(x) \in H^1([0,\pi]),~~ \theta_t(x,0)=\theta_1(x) \in L^2([0,\pi]).
 \end{equation}
We consider one of the following boundary conditions:
 \begin{equation}\label{BC}
\begin{aligned}
&u(0,t)=u(\pi,t)=0,\\
&-\iota_1 \theta(0,t)+\iota_2\theta_x(0,t)=\iota_3 \theta(\pi,t)+\iota_4\theta_x(\pi,t)=0,
\end{aligned}
\end{equation}
or
 \begin{equation}\label{BC2}
\begin{aligned}
&(u_x+\theta_t)(0,t)=(u_x+\theta_t)(\pi,t)=0,\\
&-\iota_1 \theta(0,t)+\iota_2\theta_x(0,t)=\iota_3 \theta(\pi,t)+\iota_4\theta_x(\pi,t)=0,
\end{aligned}
\end{equation}
where $\iota_{1}$ to $\iota_4$ are nonnegative constants satisfying $\iota_{1}^2+\iota_{2}^2>0$ and $\iota_{3}^2+\iota_{4}^2>0$. The functions $u_0(x)$, $\theta_0(x)$ and $\theta_1(x)$ satisfy the corresponding compatibility conditions at $0$ and $\pi$, and two additional conditions:
\begin{align}\label{1.5}
u_{0}'(x)+\theta_1(x)\in C^\alpha([0,\pi]),
\end{align}
for some $\alpha\in(0,1/4)$, and $\theta_0$  are absolutely continuous.

For any given time $T\in(0,\infty)$,
the initial-boundary value problem \eqref{1.1}-\eqref{1.5} admits a weak solution $(u(x,t), \theta(x,t))$ defined on $[0,\pi]\times[0,T]$ in the sense of Definition \ref{def}. Moreover, the associated energy
\begin{align}\label{1.13a}
\mathcal{E}(t):=\fr{1}{2}\int_{0}^\pi(\theta_{t}^2+c^2(\theta)\theta_{x}^2+u^2)(x,t)\ {\rm d}x+B_\pi(\theta(\pi,t))+B_0(\theta(0,t)),
\end{align}
is well-defined for $t\in[0,T]$ and satisfies
\begin{align}\label{1.13b}
\mathcal{E}(t)\leq \mathcal{E}(0) -\int_{0}^t\int_{0}^\pi((u_x+\theta_t)^2+\theta_{t}^2)(x,t)\ {\rm d}x{\rm d}t,
\end{align}
where $B_0(\theta(0,t))$ and $B_\pi(\theta(\pi,t))$ are, respectively, the boundary energies on $x=0$ and $x=\pi$, defined as
\begin{equation}\label{1.13c}
B_0(\theta(0,t))=
\left\{
\begin{aligned}
\fr{\iota_1}{\iota_2}\int_{0}^{\theta(0,t)}c^2(s)s\ {\rm d}s,\ &\iota_2\neq0,\\
0, \  &\iota_2=0,
\end{aligned}
\right.
\
B_\pi(\theta(\pi,t))=
\left\{
\begin{aligned}
\fr{\iota_3}{\iota_4}\int_{0}^{\theta(\pi,t)}c^2(s)s\ {\rm d}s,\ &\iota_4\neq0,\\
0, \  &\iota_4=0.
\end{aligned}
\right.
\end{equation}
\end{theorem}

\bigskip

Next, we define the weak solution. First, the energy equality for the smooth solution is
\begin{align}\label{1.6a}
\frac{\rm d}{{\rm d}t}\int_0^\pi \frac12 u^2+\frac12 \theta_t^2+\frac12c^2\theta_x^2{\rm d}x+\int_0^\pi \theta_t^2+(u_x+\theta_t)^2 {\rm d}x-\big[(u_x+\theta_t) u+c^2\theta_t\theta_x\big]\Big|_{x=0}^{x=\pi}=0, \nonumber \\
\frac{\rm d}{{\rm d}t}\left[\int_0^\pi \frac12 u^2+\frac12 \theta_t^2+\frac12c^2\theta_x^2{\rm d}x +B_\pi(\theta(\pi,t))+B_0(\theta(0,t))\right]+\int_0^\pi \theta_t^2+(u_x+\theta_t)^2 {\rm d}x=0.
\end{align}

\begin{definition}\label{def}\em{
For any given time $T>0$, we say $(u(x,t), \theta(x,t))$, defined for all $(x,t)\in[0,\pi]\times[0,T]$, is a weak solution to the initial-boundary value problem \eqref{1.1}-\eqref{1.5} if

(i) there hold
\begin{align}\label{1.7}
\int_{0}^T\int_{0}^\pi\bigg(u\psi_t-(u_x+\theta_t)\psi_x\bigg)\ {\rm d}x{\rm d}t+\int_0^T [(u_x+\theta_t)\psi]\bigg|_{x=0}^{x=\pi} {\rm d}t=0,
\end{align}
and
\begin{align}\label{1.8}
\int_{0}^T\int_{0}^\pi\bigg(\theta_t\varphi_t-(c(\theta)\varphi)_xc(\theta)\theta_x-2\theta_t\varphi -u_x\varphi\bigg)\ {\rm d}x{\rm d}t +\int_{0}^T(c^2\varphi\theta_x)\bigg|_{x=0}^{x=\pi}\ {\rm d}t=0,
\end{align}
for any test functions $\psi, \varphi\in\mathcal{F}$, where
\begin{align}\label{1.9}
\mathcal{F}:=\bigg\{f\in C^\infty((0,\pi)\times(0,T)):\ \pa_{t}^i\pa_{x}^jf\bigg|_{t=0,T}=0,\ \ \forall\ i,j=0,1,2\cdots \bigg\},
\end{align}
\begin{align}\label{1.10}
\theta\in C^{1/2}([0,\pi]\times[0,T])\cap L^2([0,T],H^1([0,\pi])),
\end{align}
and
\begin{align}
u&\in L^\infty([0,T],H^1([0,\pi]))\cap L^\infty([0,\pi]\times[0,T]),\label{1.11} \\
u_t&\in L^2([0,T],H^{-1}([0,\pi])); \label{1.12}
\end{align}

(ii) the first and second equations for initial conditions in \eqref{1.2} are satisfied pointwise, and the third equation holds in $L^p$ for $p\in[1,2)$;

(iii) the boundary conditions in \eqref{BC} or in \eqref{BC2} are satisfied in $L^2(0,T)$ sense.}
\end{definition}

We note that the global existence theory for \eqref{1.1}, is base on earlier work of Bressan-Zheng in \cite{Bressan-Zheng} on the variational wave equation. In fact, before considering \eqref{1.1}, a class of simplified 1-d wave models only on $\nb$  were first studied. For example, when $\nb=(\cos \theta, 0,\sin\theta)(x,t)$, $x\in \mathbb{R}$,  the {\em variational wave equation} satisfies
\begin{equation}\label{VW}
\theta_{tt}-c(\theta)\bigl(c(\theta)\theta_x\bigr)_x=0, \quad \theta(x,0)=\theta_0(x)\in H^1,\quad
\theta_t(x,0)=\theta_1(x)\in L^2.
\end{equation}
It is natural to consider the finite energy ($H^1$) initial data. Since $H^1\hookrightarrow C^{1/2}$ in 1-d, which is not Lipschitz, finite time cusp singularity forms even if initial data are smooth  \cite{GHZ}. For the Cauchy problem \eqref{VW}, the existence, uniqueness and Lipschitz continuity of global  $C^{1/2}$ energy conservative solutions  was established by \cite{Bressan-Zheng, HR}, \cite{BCZ2} and \cite{BC,BC2015}, respectively. Later, these results were extended to a wave system with $\nb\in{\mathbb S}^2$ when the Oseen-Frank potential takes its general form in \cite{CZZ12,CCD,CCS, ZZ10, ZZ11}, and another general system in \cite{H}. See existence of dissipative solution of \eqref{VW} for the Cauchy problem at \cite{ZZ03,BH}.

Our existence result on the initial boundary value problem also applies to the variational wave equation \eqref{VW}, which is missing for many years.

However, one cannot easily extend the result for the variational wave equation to  \eqref{1.1} (or \eqref{sys}). This is because $u_x$ has the similar regularity as $\theta_t$, which maybe unbounded. There is no direct method to cope with the variational wave equation with an unbounded source term. The key observation in \cite{CHL20} which helps solving this issue, is to find $J=u_x+\theta_t\in C^\alpha\cap L^2\cap L^\infty$ on $(x,t)$ for \eqref{1.1}.

For the initial-boundary value problem, we still use $J$ to rewrite \eqref{1.1}. Then, for any given $J\in C^\alpha\cap L^2\cap L^\infty$, we first establish the global existence result on $\theta(x,t)\in L^\infty(H^1)$, by solving an initial-boundary value problem of \eqref{VW} with damping and the force term $J$. In this step, we need to change the equation into a semilinear system on characteristic coordinates. This method was first used by Bressan-Zheng in \cite{Bressan-Zheng}, then by Chen-Huang-Liu in \cite{CHL20}, for the Cauchy problem. For the boundary value problem, we need to transform the boundary conditions on $(u, \theta)$ in the original $(x,t)$-coordinates to conditions on $J$ and other new dependent variables under the new coordinates. Here we need some new methods to cope with these additional boundary values (see Subsection \ref{S2.5}).

The second step is to find a fixed point using the heat equation on $J$ for the boundary value problem. The difficulty lies in the fact that the source term of the equation on $J$ is only $H^{-1}$. One cannot directly use the smoothing effect of the heat equation, but need to first use the wave equation to change $\theta_{tt}$ into $\theta_{xx}$ and other lower order terms. Then one can show the enhanced regularity on $J$, since the smoothing effect of heat equation on the source term $\theta_{xx}$ is much better than $\theta_{tt}$. Different from the Cauchy problem \cite{CHL20}, in this initial-boundary value problem, the heat kernel is expressed as an infinite series by using the image method. Finally we used the Duhamel's principle and the Schauder fixed point theorem to prove the existence of a fixed point.

%
\subsection{Structure of this paper}
We will first consider the nonslip (Dirichlet) boundary condition \eqref{BC} on velocity $u$, in sections 3 and 4. To cope with different types of boundary conditions on $\theta$, we only need some minor changes in the proof. To avoid a repeat, we propose to only prove the following mixed boundary conditions.
\begin{equation}\label{1.3}
\begin{aligned}
&u(0,t)=u(\pi,t)=0,\\
&\theta(0,t)=\iota \theta(\pi,t)+\theta_x(\pi,t)=0,
\end{aligned}
\end{equation}
where $\iota\geq 0$. This is a special case of \eqref{BC}.
The functions $u_0(x), \theta_0(x)$ and $\theta_1(x)$ satisfy the corresponding compatibility conditions at $0$ and $\pi$
\begin{equation}\label{1.4}
\begin{aligned}
\theta_0(0)=\iota\theta_0(\pi)+\theta_1(\pi)=0,\quad u_0(0)=u_0(\pi)=0.
\end{aligned}\end{equation}

Correspondingly, our results can be stated as follows.
\begin{theorem}
\label{thm}
Assume all conditions in Theorem \ref{thm_0} hold. For any given time $T\in(0,\infty)$,
the initial-boundary value problem \eqref{1.1}-\eqref{1.5} admits a weak solution $(u(x,t), \theta(x,t))$ defined on $[0,\pi]\times[0,T]$ in the sense of Definition \ref{def}. Moreover, the associated energy
\begin{align}\label{1.13}
\mathcal{E}(t):=\fr{1}{2}\int_{0}^\pi(\theta_{t}^2+c^2(\theta)\theta_{x}^2+u^2)(x,t)\ {\rm d}x +B(\theta(\pi,t)),
\end{align}
with $B(\theta(\pi,t))=\int_{0}^{\theta(\pi,t)}\iota c^2(s)s\ {\rm d}s\geq0$, is well-defined for $t\in[0,T]$ and satisfies
\begin{align}\label{1.14}
\mathcal{E}(t)\leq \mathcal{E}(0) -\int_{0}^t\int_{0}^\pi((u_x+\theta_t)^2+\theta_{t}^2)(x,t)\ {\rm d}x{\rm d}t.
\end{align}
\end{theorem}
In the proof of Theorem \ref{thm}, we include all necessary techniques to cope with three types of  (Dirichlet, Neumann and Robin)  boundary conditions on $\theta$ included in \eqref{BC}. So all other cases in Theorem \ref{thm_0} with the boundary condition \eqref{BC} can be proved similarly. We leave them to the reader.

In section 5, we will present a proof of Theorem \ref{thm_0} with stress-free (Neumann) boundary condition \eqref{BC2}.

\begin{remark}\em{
The result is still correct for the general Dirichlet and Neumann boundary on $u$ and $\theta$, with for example, $u(0,t)=f_1(t)$, $u(\pi,t)=f_2(t)$, $\theta(0,t)=g_1(t)$ and $\iota\theta(\pi,t)+\theta_x(\pi,t)=g_2(t)$ with absolutely continuous $f_{1,2}(t), g_{2}(t)$ functions and $C^1$ continuous $g_1(t)$ function.
For this case, when considering the global existence of the wave equations, we replace the boundary conditions on $L_0$ and $L_\pi$ in \eqref{2.23} by
\begin{align*}
\begin{array}{ll}
\ds \fr{z}{2}=\arctan[2g_{1}'(t)-\tan\frac w2],\   q-\frac{1+[2g_{1}'(t)-\tan\frac w2]^2}{1+\tan^2\frac w2}q=0, \  &{\rm on}\ L_0, \\[8pt]
\ds \frac w2=\arctan[\tan\frac z2 +2c(\theta)g_2(t)-2\iota c(\theta)\theta],\  p-\frac{1+[\tan\frac z2+2c(\theta)g_2(t)-2\iota c(\theta)\theta]^2}{1+\tan^2\frac z2}q=0, \  &\text{on } L_\pi,
\end{array}
\end{align*}
and replace the boundary energy $B(\theta(\pi,t))$ in \eqref{1.13} by a corresponding form. Moreover,
instead of \eqref{3.1}, \eqref{3.2}, the current initial-boundary value problem for the variable $J=u_x+\theta_t$ is
\begin{align}\label{1.16}
\left\{
\begin{array}{ll}
\ds J_t-J_{xx}=c(\theta)(c(\theta)\theta_x)_x-\theta_t-J, \\
\ds J(x,0)=J_0(x), \\
\ds J_x(0,t)=f_{1}'(t),\ \ J_x(\pi,t)=f_{2}'(t).
\end{array}
\right.
\end{align}
Furthermore, replaced \eqref{3.50}, the initial-boundary value problem for the variable $u$ now is
\begin{align}\label{1.17}
\left\{
\begin{array}{ll}
\ds u_t-u_{xx}=\theta_{tx}, \\
\ds u(x,0)=u_0(x), \\
\ds u(0,t)=f_{1}(t),\ \ u(\pi,t)=f_{2}(t).
\end{array}
\right.
\end{align}
We can introduce some suitable variables for problems \eqref{1.16} and \eqref{1.17} to transform them into homogeneous boundary condition problems, and then use the Green/Neumann functions defined in Section \ref{S3} to express the corresponding weak solutions. The proof is very similar to the cases of the homogeneous boundary conditions \eqref{BC} or \eqref{BC2}.}
\end{remark}

\section{Existence of wave equation for any given $J=u_x+\theta_t$}\label{S2}

In this section, we show the global existence of H\"older continuous solutions to the initial-boundary value problem of the nonlinear wave equation in \eqref{1.1} for any given $J=u_x+\theta_t$.

\subsection{The semilinear system in characteristic coordinates}\label{S2.1}

Let $J=u_x+\theta_t$. The wave equation in \eqref{1.1} reads
\begin{align}\label{2.1}
\theta_{tt}-c(\theta)(c(\theta)\theta_x)_x+\theta_t+J=0.
\end{align}
Denote
\begin{align}\label{2.2}
R:=\theta_t+c(\theta)\theta_x,\quad S:=\theta_t-c(\theta)\theta_x,
\end{align}
so that
\begin{align}\label{2.3}
\theta_t=\fr{R+S}{2},\quad \theta_x=\fr{R-S}{2c(\theta)}.
\end{align}
By \eqref{2.1}, the equations in terms of variables $(R,S)$ are
\begin{align}\label{2.4}
\left\{
\begin{array}{l}
\dps R_t-c(\theta)R_x=\fr{c'(\theta)}{4c(\theta)}(R^2-S^2)-\fr{1}{2}(R+S)-J,\\[8pt]
\dps S_t+c(\theta)S_x=\fr{c'(\theta)}{4c(\theta)}(S^2-R^2)-\fr{1}{2}(R+S)-J.
\end{array}
\right.
\end{align}

Let $(x,t)$ be any point in $[0,\pi]\times[0,\infty)$. We define the forward and backward characteristics $x=x_\pm(s; x,t)(s\leq t)$ passing through the point $(x,t)$ as follows
\begin{align}\label{2.5}
\left\{
\begin{array}{l}
\dps \fr{{\rm d}x_\pm(s;x,t)}{{\rm d}s}=\pm c(\theta(x_\pm(s;x,t),s)), \\
\dps x_\pm(t;x,t)=x.
\end{array}
\right.
\end{align}
We now define the coordinate transformation $(x,t)\rightarrow(X,Y)$ on $[0,\pi]\times[0,\infty)$.
We first specify this transformation to transform the lines $x=0$ and $x=\pi$ with $t\geq0$ into the lines
$Y=X$ with $X\geq0$ and $Y=X-\widetilde{X}$ with $X\geq \widehat{X}$, respectively, where
\begin{align}\label{2.6}
\widetilde{X}=\int_{0}^\pi(1+R_{0}^2(z))\ {\rm d}z-\int_{\pi}^0(1+S_{0}^2(z))\ {\rm d}z,\quad \widehat{X}=\int_{0}^\pi(1+R_{0}^2(z))\ {\rm d}z.
\end{align}
Here
\begin{align}\label{2.7}
R_{0}(x)=\theta_1(x)+c(\theta_0(x))\theta_{0}'(x),\quad S_{0}(x)=\theta_1(x)-c(\theta_0(x))\theta_{0}'(x),\ \ \forall\ x\in[0,\pi].
\end{align}
Moreover, we set that the segment $t=0$ with $x\in[0,\pi]$ is transformed to a piece of curve $\Gamma_0: Y=\phi(X)(X\in[0,\widehat{X}])$ defined through a parametric $x\in[0,\pi]$
\begin{align}\label{2.8}
X=\int_{0}^x(1+R_{0}^2(z))\ {\rm d}z,\quad Y=\int_{x}^0(1+S_{0}^2(z))\ {\rm d}z.
\end{align}
It is observed by the initial data \eqref{1.2} that the two functions $X=X(x)$ and $Y=Y(x)$ with $x\in[0,\pi]$ are well defined and absolutely continuous. Moreover, $X(x)$ is strictly
increasing while $Y(x)$ is strictly decreasing. Hence the function $Y=\phi(X)$ is continuous and strictly
decreasing. Next for any point $(x,t)\in[0,\pi]\times(0,\infty)$, we draw the backward characteristic $x_-(s; x,t)$ up to a point $P_1$ on $x=\pi$, and then draw the forward characteristic $x_+(s; P_1)$ up to a point $P_2$ on $x=0$. Repeating the above process, since the wave speed $c(\theta)\geq C_L$, we can reach the segment $t=0(x\in[0,\pi])$ through finite steps by \eqref{1.6}. Assume that there exists a point $P_l$ on $x=0$ or $x=\pi$ such that the backward characteristic $x_-(s; P_l)$ or the forward characteristic $x_+(s; P_l)$ intersects the segment $t=0(x\in[0,\pi])$ at a point $P^*(x_{P^*},0)$. It is clear that the points $P_i$ are on $x=\pi$ for odd numbers $i\leq l$, and on $x=0$ for even numbers $i\leq l$. See Fig. 2.1 (a) for the illustration. Denote the coordinate of $P_i$ by $(\pi,t_i)$ if $i$ is an odd number, and by $(0,t_i)$ if $i$ is even. The numbers $t_i(i=1,\cdots,l)$ and $x_{P^*}$ can be determined sequentially as follows
\begin{align}\label{2.9}
\left\{
\begin{array}{l}
x_-(t_1;x,t)=\pi,\\
x_+(t_2;P_1)=0,\\
x_-(t_3;P_2)=\pi,\\
\vdots \\
x_+(t_l;P_{l-1})=0,\ \ x_{P^*}=x_-(0;P_l),\ \ \text{ if $l$ is an even number},\\
x_-(t_l;P_{l-1})=\pi,\ \ x_{P^*}=x_+(0;P_l),\ \ \text{ if $l$ is an odd number}.
\end{array}
\right.
\end{align}
Then we define the value of $X(x,t)$ by
\begin{align}\label{2.10}
X(x,t)=&X(P_1)=Y(P_1)+\widetilde{X}=Y(P_2)+\widetilde{X}=X(P_2)+\widetilde{X} \nonumber \\
=&X(P_3)+\widetilde{X}=(Y(P_3)+\widetilde{X})+\widetilde{X} =Y(P_4)+2\widetilde{X} \nonumber \\
=&\cdots=
\left\{
\begin{array}{l}
X(P_l)+k\widetilde{X},\qquad \quad \ l=2k\\
Y(P_l)+(k+1)\widetilde{X},\ \ l=2k+1
\end{array}
\right.
\nonumber \\[6pt]
=&
\left\{
\begin{array}{l}
\dps \int_{0}^{x_{P^*}}(1+R_{0}^2(z))\ {\rm d}z +k\widetilde{X}, \qquad  \ \  l=2k, \\[8pt]
\dps \int_{x_{P^*}}^0(1+S_{0}^2(z))\ {\rm d}z +(k+1)\widetilde{X}, \ \ l=2k+1.
\end{array}
\right.
\end{align}

\begin{figure}[htbp]
\begin{center}
\begin{minipage}[t]{0.38\textwidth}
\includegraphics[scale=0.58]{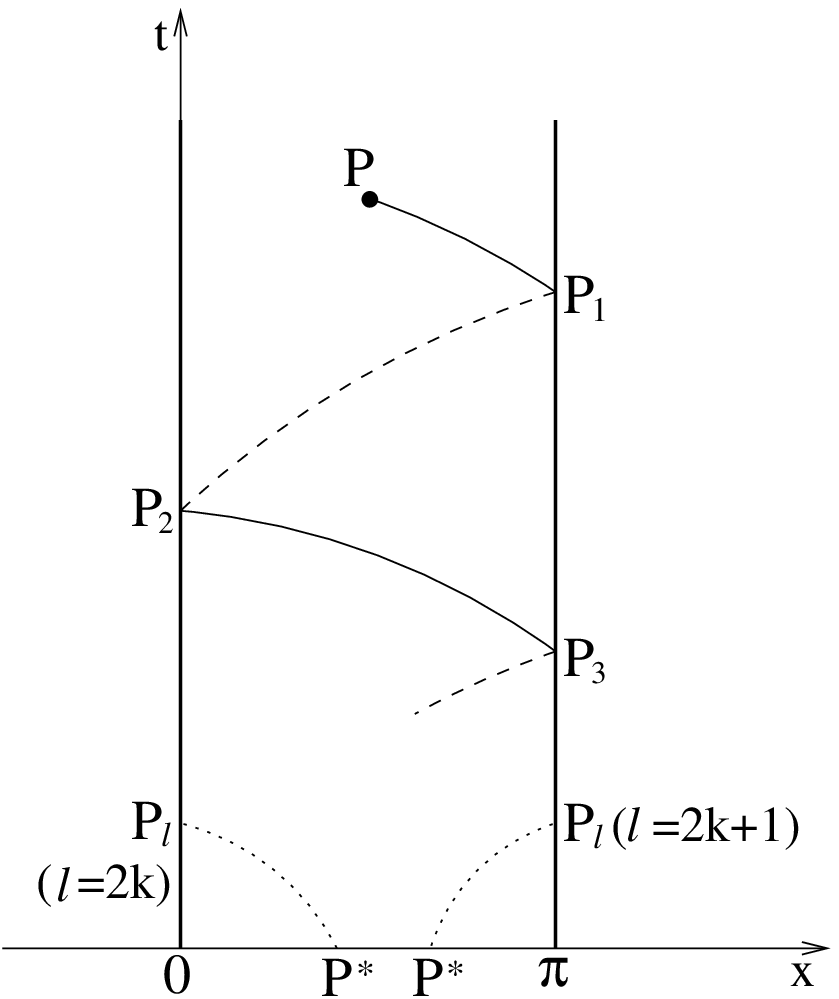}
{\center \qquad  \qquad \qquad   (a)}
\end{minipage}
\begin{minipage}[t]{0.38\textwidth}
\includegraphics[scale=0.58]{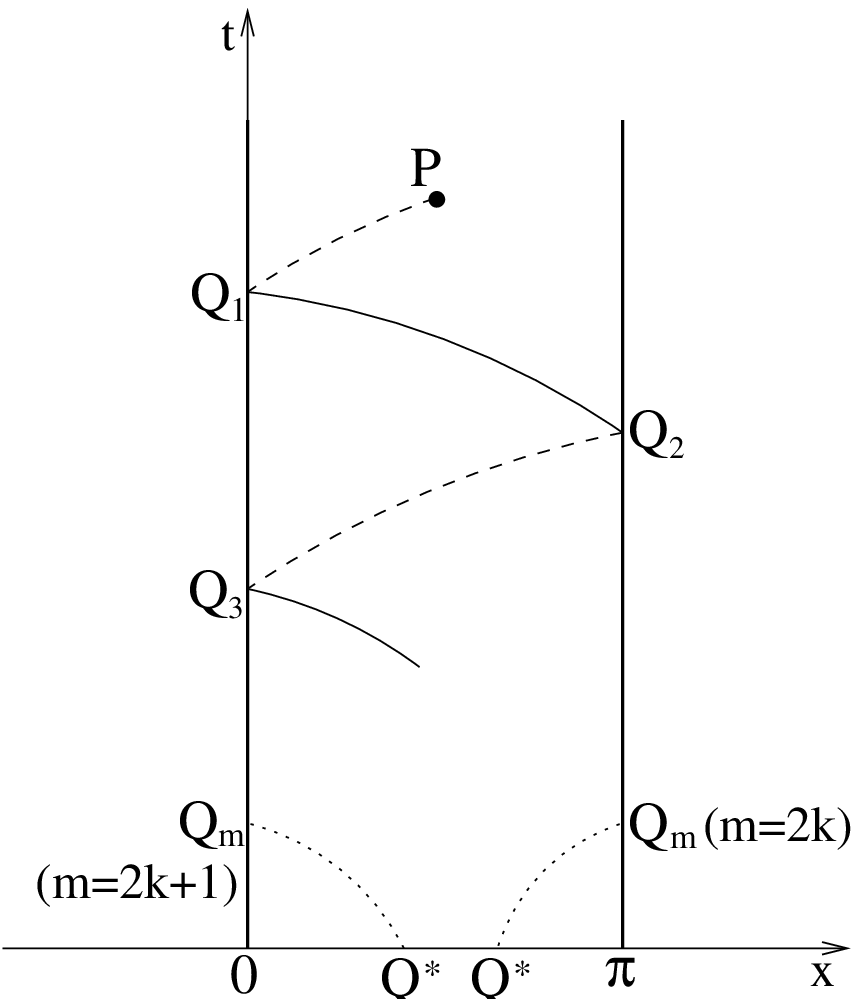}
{\center \qquad  \qquad \qquad  (b)}
\end{minipage}
\caption{Characteristic curves. }
\end{center}
\end{figure}

To define the value of $Y(x,t)$, we draw the forward characteristic $x_+(s; x,t)$ up to a point $Q_1$ on $x=0$, and then draw the backward characteristic $x_-(s; Q_1)$ up to a point $Q_2$ on $x=\pi$. Similarly, we can reach the segment $t=0(x\in[0,\pi])$ through finite steps by repeating the above process.
Assume that there exists a point $Q_m$ on $x=0$ or $x=\pi$ such that the backward characteristic $x_-(s; Q_m)$ or the forward characteristic $x_+(s; Q_m)$ intersects the segment $t=0(x\in[0,\pi])$ at a point $Q^*(x_{Q^*},0)$. Obviously, the points $Q_i$ are on $x=0$ for odd numbers $i\leq m$, and on $x=\pi$ for even numbers $i\leq m$. See Fig. 2.1 (b) for the illustration. Denote the coordinate of $Q_i$ by $(0,\tilde{t}_i)$ if $i$ is an odd number, and by $(\pi,\tilde{t}_i)$ if $i$ is even. We can determine
the numbers $\tilde{t}_i(i=1,\cdots,m)$ and $x_{Q^*}$ sequentially
\begin{align}\label{2.11}
\left\{
\begin{array}{l}
x_+(\tilde{t}_1;x,t)=0,\\
x_-(\tilde{t}_2;Q_1)=\pi,\\
x_+(\tilde{t}_3;Q_2)=0,\\
\vdots \\
x_-(\tilde{t}_m;Q_{m-1})=\pi,\ \ x_{Q^*}=x_+(0;Q_m),\ \ \text{ if $m$ is an even number},\\
x_+(\tilde{t}_m;Q_{m-1})=0,\ \ x_{Q^*}=x_-(0;Q_m),\ \  \text{ if $m$ is an odd number}.
\end{array}
\right.
\end{align}
Thus the value of $Y(x,t)$ can be defined as follows
\begin{align}\label{2.12}
Y(x,t)=&Y(Q_1)=X(Q_1)=X(Q_2)=Y(Q_2)+\widetilde{X}=Y(Q_3)+\widetilde{X} \nonumber \\
=&X(Q_3)+\widetilde{X} =X(Q_4)+\widetilde{X}=(Y(Q_4)+\widetilde{X})+ \widetilde{X}=Y(Q_4)+2\widetilde{X} \nonumber \\
=&\cdots=
\left\{
\begin{array}{l}
Y(Q_m)+k\widetilde{X},\ \ m=2k\\
X(Q_m)+k\widetilde{X},\ \ m=2k+1
\end{array}
\right.
\nonumber \\[6pt]
=&
\left\{
\begin{array}{l}
\dps \int_{x_{Q^*}}^0(1+S_{0}^2(z))\ {\rm d}z +k\widetilde{X}, \ \ \ \ m=2k, \\[12pt]
\dps \int_{0}^{x_{Q^*}}(1+R_{0}^2(z))\ {\rm d}z +k\widetilde{X}, \ \  m=2k+1.
\end{array}
\right.
\end{align}
It is easy to note by \eqref{2.9} and \eqref{2.11} that for the point $(x,t)$ on $t=0(x\in[0,\pi])$, the transformation defined in \eqref{2.10}, \eqref{2.12} reduces to \eqref{2.8}. Furthermore, according to the construction process of the transformation $(x,t)\rightarrow(X,Y)$, we see that $X$ and $Y$ are constants along backward and forward characteristic, respectively; that is,
\begin{align}\label{2.13}
X_t-c(\theta)X_x=0,\quad Y_t+c(\theta)Y_x=0,
\end{align}
from which one has for any smooth function $f$
\begin{align}\label{2.14}
\begin{array}{l}
f_t+c(\theta)f_x=(f_XX_t+f_YY_t)+c(\theta)(f_XX_x+f_YY_x)=2cX_xf_X, \\
f_t-c(\theta)f_x=(f_XX_t+f_YY_t)-c(\theta)(f_XX_x+f_YY_x)=-2cY_xf_Y.
\end{array}
\end{align}

For convenience to deal with possibly unbounded values of $R$ and $S$, one can introduce the variables
\begin{align}\label{2.15}
w:=2\arctan R,\quad z:=2\arctan S.
\end{align}
In order to complete the system, we further introduce two key dependent variables
\begin{align}\label{2.16}
p:=\fr{1+R^2}{X_x},\quad q:=\fr{1+S^2}{-Y_x}.
\end{align}
Then by summing \eqref{2.4}, \eqref{2.13}-\eqref{2.16}, we acquire a semilinear hyperbolic system
with smooth coefficients for the variables $\theta, w, z, p, q$ in terms of the coordinates $(X,Y)$
\begin{align}\label{2.17}
\begin{array}{l}
\dps \theta_X=\fr{\sin w}{4c}p,\quad \theta_Y=\fr{\sin z}{4c}q, \\[8pt]
\dps w_Y=\fr{q}{4c}\bigg\{\fr{c'}{c}\bigg(\cos^2\fr{z}{2}-\cos^2\fr{w}{2}\bigg)-\sin w\cos^2\fr{z}{2} -\sin z\cos^2\fr{w}{2} -4J\cos^2\fr{w}{2}\cos^2\fr{z}{2}\bigg\}, \\[8pt]
\dps z_X=\fr{p}{4c}\bigg\{\fr{c'}{c}\bigg(\cos^2\fr{w}{2}-\cos^2\fr{z}{2}\bigg)-\sin w\cos^2\fr{z}{2} -\sin z\cos^2\fr{w}{2} -4J\cos^2\fr{w}{2}\cos^2\fr{z}{2}\bigg\}, \\[8pt]
\dps p_Y=\fr{pq}{2c}\bigg\{\fr{c'}{4c}(\sin z-\sin w)-\fr{1}{4}\sin w\sin z -\sin^2\fr{w}{2}\cos^2\fr{z}{2} -J\sin w\cos^2\fr{z}{2}\bigg\},
\\[8pt]
\dps q_X=\fr{pq}{2c}\bigg\{\fr{c'}{4c}(\sin w-\sin z)-\fr{1}{4}\sin w\sin z -\sin^2\fr{z}{2}\cos^2\fr{w}{2} -J\sin z\cos^2\fr{w}{2}\bigg\}.
\end{array}
\end{align}
The detailed derivation of \eqref{2.17} can be found in Chen-Huang-Liu \cite{CHL20}. In addition, one also has
\begin{align}\label{2.18}
\left\{
\begin{array}{l}
\dps x_X=\fr{1}{2X_x}=\fr{1+\cos w}{4}p, \\[8pt]
\dps x_Y=\fr{1}{2Y_x}=-\fr{1+\cos z}{4}q,
\end{array}
\right.
\quad
\left\{
\begin{array}{l}
\dps t_X=\fr{1}{2cX_t}=\fr{1+\cos w}{4c}p, \\[8pt]
\dps t_Y=-\fr{1}{2cY_t}=\fr{1+\cos z}{4c}q,
\end{array}
\right.
\end{align}
which can be achieved by setting $f=x$ and $f=t$ in \eqref{2.14}. It suggests by \eqref{2.18} that
\begin{align}\label{2.19}
{\rm d}x{\rm d}t=\fr{pq}{2c}\cos^2\fr{w}{2}\cos^2\fr{z}{2}\ {\rm d}X{\rm d}Y.
\end{align}

\subsection{The boundary value problem in the $(X,Y)$ coordinates}\label{S2.2}

According to the construction of the coordinate transformation $(x,t)\rightarrow(X,Y)$, we know that
the segment $t=0(x\in[0,\pi])$ is transformed into a piece of continuous and strictly
decreasing curve $\Gamma_0: Y=\phi(X)(X\in[0,\widehat{X}])$, which defines in \eqref{2.8} by a parametric $x\in[0,\pi]$. Moreover, the lines $x=0(t\geq0)$ and $x=\pi(t\geq0)$ are transformed into the lines $L_0: Y=X(X\geq0)$ and $L_\pi: Y=X-\widetilde{X}(X\geq \widehat{X})$, respectively. See Fig. 2.2.

\begin{figure}[htbp]
\centering
\includegraphics[scale=0.58]{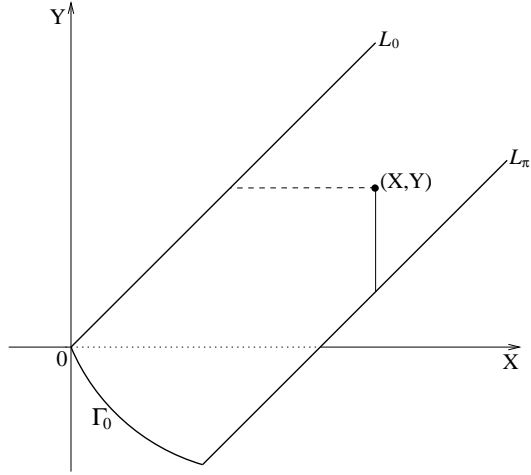}
\caption{The region in the $(x,Y)$ plane.}
\end{figure}

Since $\Gamma_0$ is parameterized by the parameter $x$, we can thus assign the
boundary data $(\bar{\theta}, \bar{w}, \bar{z}, \bar{p}, \bar{q})\in L^\infty$ defined by
\begin{align}\label{2.20}
\bar{\theta}=\theta_0(x),\ \bar{w}=2\arctan R_0(x),\ \bar{z}=2\arctan S_0(x),\ \bar{p}=1,\ \bar{q}=1,
\end{align}
where $R_0(x)$ and $S_0(x)$ are given in \eqref{2.7} and the boundary values $(p,q)$ come from \eqref{2.16}. Moreover, recalling \eqref{1.3} yields $\theta_t=0$ on $x=0$, which implies by \eqref{2.3} and \eqref{2.15} that $R+S=0$ and $w+z=0$ on $x=0$. On $x=\pi$, $\theta_x=-\iota\theta$ leads to $R-S=\tan\frac w2-\tan\frac z2=-2\iota c(\theta)\theta$.

Hence the boundary values of $(w,z)$ satisfy
\begin{align}\label{2.21}
w+z=0, \ \ {\rm on}\ L_0, \quad \frac w2=\arctan[\tan\frac z2-2\iota c(\theta)\theta], \ \ \text{ on } L_\pi.
\end{align}
Furthermore, in view of \eqref{2.18}, we see that
$$
{\rm d}x=x_X{\rm d}X+x_Y{\rm d}Y=\fr{1+\cos w}{4}p{\rm d}X- \fr{1+\cos z}{4}q{\rm d}Y,
$$
which indicates by \eqref{2.21} and the definitions of $L_0, L_\pi$ that
$$
0=\fr{1+\cos w}{4}(p-q){\rm d}X, \ \ {\rm on}\ \  L_0, \quad 0=[(1+\cos w)p-(1+\cos z)q]{\rm d}X, \ \ {\rm on}\ \  L_\pi.
$$
Thus the boundary values of $(p,q)$ satisfy
\begin{align}\label{2.22}
p-q=0, \ \ {\rm on}\ L_0, \quad p-\frac{1+(\tan\frac z2-2\iota c(\theta)\theta)^2}{1+\tan^2\frac z2}q=0, \text{ on } L_\pi.
\end{align}
Summing up \eqref{2.20}, \eqref{2.21} and \eqref{2.22}, the new boundary value problem in the $(X,Y)$ coordinate plane is the semilinear system \eqref{2.17} supplemented with
\begin{align}\label{2.23}
\begin{array}{ll}
(\theta, w, z, p, q)=(\bar{\theta}, \bar{w}, \bar{z}, \bar{p}, \bar{q}),  \ \ &{\rm on}\ \Gamma_0, \\
w+z=0,\ \  p-q=0, \ \ &{\rm on}\ L_0, \\
\ds \frac w2=\arctan[\tan\frac z2-2\iota c(\theta)\theta],\ \ p-\frac{1+(\tan\frac z2-2\iota c(\theta)\theta)^2}{1+\tan^2\frac z2}q=0, \ \ &\text{on } L_\pi
\end{array}
\end{align}
We use $\Omega$ to denote the region bounded by $t=0, x=0$ and $x=\pi$ in the $(x,t)$ plane, and use $\widetilde{\Omega}$ to represent its image in the $(X,Y)$ plane which is bounded by $\Gamma_0, L_0$ and $L_\pi$.

\subsection{Local existence of the boundary value problem to the semilinear system}
\label{S2.3}

We use the level lines of $X$ and $Y$ to divide the region $\widetilde{\Omega}$ into a series of subregions
\[
\widetilde{\Omega}=\bigcup_{n=0}^\infty\Omega^{n},
\]
where $\Omega^0=\Omega_{1}^0\cup\Omega_{2}^0\cup\Omega_{3}^0$ with
\begin{align*}
\begin{array}{l}
\Omega_{1}^0=\{(X,Y):\ 0<X\leq \widehat{X},\ 0<Y\leq X\},\\
\Omega_{2}^0=\{(X,Y):\ 0\leq X\leq \widehat{X},\ \phi(X)\leq Y\leq 0\},\\
\Omega_{3}^0=\{(X,Y):\ \widehat{X}<X\leq \widetilde{X},\ X-\widetilde{X} \leq Y\leq 0\},
\end{array}
\end{align*}
and $\Omega^n=\Omega_{1}^n\cup\Omega_{2}^n\cup\Omega_{3}^n$ with
\begin{align*}
\begin{array}{l}
\Omega_{1}^n=\{(X,Y):\ \widehat{X}+k\widetilde{X}<X\leq (k+1)\widetilde{X},\ \widehat{X}+k\widetilde{X}<Y\leq X\},\\
\Omega_{2}^n=\{(X,Y):\ \widehat{X}+k\widetilde{X}<X\leq (k+1)\widetilde{X},\ k\widetilde{X}<Y\leq \widehat{X}+k\widetilde{X}\},\\
\Omega_{3}^n=\{(X,Y):\ (k+1)\widetilde{X}<X\leq \widehat{X}+(k+1)\widetilde{X},\ X-\widetilde{X}\leq Y\leq \widehat{X}+k\widetilde{X}\},
\end{array}
\end{align*}
for $n=2k+1(k=0,1,2,\cdots)$, and with
\begin{align*}
\begin{array}{l}
\Omega_{1}^n=\{(X,Y):\ k\widetilde{X}<X\leq \widehat{X} +k\widetilde{X},\ k\widetilde{X}<Y\leq X\},\\
\Omega_{2}^n=\{(X,Y):\ k\widetilde{X}<X\leq \widehat{X}+k\widetilde{X},\ \widehat{X}+(k-1)\widetilde{X}<Y\leq k\widetilde{X}\},\\
\Omega_{3}^n=\{(X,Y):\ \widehat{X}+k\widetilde{X}<X\leq (k+1)\widetilde{X},\ X-\widetilde{X}\leq Y\leq k\widetilde{X}\},
\end{array}
\end{align*}
for $n=2k(k=1,2,3,\cdots)$. See Fig. 2.3 for the illustration.

\begin{figure}[htbp]
\centering
\includegraphics[scale=0.58]{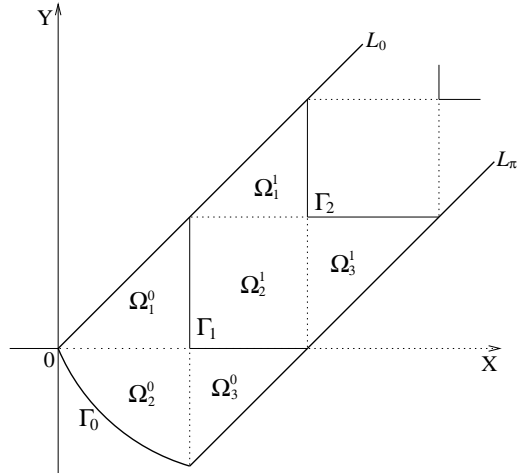}
\caption{The region $\Omega_{i}^n$.}
\end{figure}

According to the local existence result of the Cauchy problem in \cite{CHL20}, we know that there exists a small positive number $\delta_2<\min\{\widehat{X}, -\phi(\widehat{X})\}$ such that the problem \eqref{2.17} with $(\theta, w, z, p, q)|_{\Gamma_0}=(\bar{\theta}, \bar{w}, \bar{z}, \bar{p}, \bar{q})$ admits a local solution $(\theta, w, z, p, q)(X,Y)$ on $\Omega_{2\delta_2}^0$, where
\begin{align}\label{2.24}
\Omega_{2\delta_2}^0=\{(X,Y)\in\Omega_{2}^0:\ {\rm dist}((X,Y),\Gamma_0)\leq\delta_2\}.
\end{align}
We next only discuss the local existence of solutions in the region $\Omega_{1}^0$ near point $(0,0)$, the local existence result in the region $\Omega_{3}^0$ near point $(\widehat{X},\phi(\widehat{X}))$ can be obtained similarly.

Denote $\Omega_{1\delta_2}^0=\Omega_{1}^0\cap\{X<\delta_2\}$. By means of the solution on $Y=0(X\in[0,\delta_2])$, we can use \eqref{2.17}, \eqref{2.18} and \eqref{2.23} to construct a map
for any point $(X,Y)\in\Omega_{1\delta_2}^0$,
\begin{align}\label{2.25}
(\hat{\theta}, \hat{w}, \hat{z}, \hat{p}, \hat{q})=\mathcal{T}_1(\theta, w, z, p, q),
\end{align}
where
\begin{align}\label{2.26}
\hat{\theta}(X,Y)=\theta(X,0)+\int_{0}^{Y}\fr{\sin z}{4c}q(X,Y')\ {\rm d}Y',
\end{align}
\begin{align}\label{2.27}
\hat{w}(X,Y)=&w(X,0)+\int_{0}^{Y}\fr{q}{4c}\bigg\{\fr{c'}{c}\bigg(\cos^2\fr{z}{2}-\cos^2\fr{w}{2}\bigg)-\sin w\cos^2\fr{z}{2} \nonumber \\
&-\sin z\cos^2\fr{w}{2} -4J(x_m,t_m)\cos^2\fr{w}{2}\cos^2\fr{z}{2}\bigg\}(X,Y')\ {\rm d}Y',
\end{align}
\begin{align}\label{2.28}
\hat{z}(X,Y)=&-w(Y,0)-\int_{0}^{Y}\fr{q}{4c}\bigg\{\fr{c'}{c}\bigg(\cos^2\fr{z}{2}-\cos^2\fr{w}{2}\bigg)-\sin w\cos^2\fr{z}{2}-\sin z\cos^2\fr{w}{2} \nonumber \\
&-4J(x_m,t_m)\cos^2\fr{w}{2}\cos^2\fr{z}{2}\bigg\}(Y,Y')\ {\rm d}Y' + \int_{Y}^{X}\fr{p}{4c}\bigg\{\fr{c'}{c}\bigg(\cos^2\fr{w}{2}-\cos^2\fr{z}{2}\bigg) \nonumber \\
& -\sin w\cos^2\fr{z}{2}-\sin z\cos^2\fr{w}{2} -4J(x_n,t_n)\cos^2\fr{w}{2}\cos^2\fr{z}{2}\bigg\} (X',Y)\ {\rm d}X',
\end{align}
\begin{align}\label{2.29}
\hat{p}(X,Y)=&p(X,0)+\int_{0}^{Y}\fr{pq}{2c}\bigg\{\fr{c'}{4c}(\sin z-\sin w)-\fr{1}{4}\sin w\sin z -\sin^2\fr{w}{2}\cos^2\fr{z}{2} \nonumber \\
&-J(x_m,t_m)\sin w\cos^2\fr{z}{2}\bigg\}(X,Y')\ {\rm d}Y',
\end{align}
and
\begin{align}\label{2.30}
\hat{q}(X,Y)=&p(Y,0)+\int_{0}^{Y}\fr{pq}{2c}\bigg\{\fr{c'}{4c}(\sin z-\sin w)-\fr{1}{4}\sin w\sin z -\sin^2\fr{w}{2}\cos^2\fr{z}{2} \nonumber \\
&-J(x_m,t_m)\sin w\cos^2\fr{z}{2}\bigg\}(Y,Y')\ {\rm d}Y' +\int_{Y}^{X} \fr{pq}{2c}\bigg\{\fr{c'}{4c}(\sin w-\sin z) \nonumber \\
&-\fr{1}{4}\sin w\sin z -\sin^2\fr{z}{2}\cos^2\fr{w}{2} -J(x_n,t_n)\sin z\cos^2\fr{w}{2}\bigg\}(X',Y)\ {\rm d}X'.
\end{align}
The points $(x_m,t_m)$ and $(x_n,t_n)$ are defined as follows
\begin{align}\label{2.31}
x_m(Z,Y')=\int_{Y'}^Z-\fr{1+\cos w}{4}p(X',Y')\ {\rm d}X',\ \ (Z=X,Y),
\end{align}
\begin{align}\label{2.32}
t_m(Z,Y')=&t(Y',\phi(Y'))+\int_{\phi(Y')}^{Y'}\fr{1+\cos z}{4c}q(Y',Y'')\ {\rm d}Y'' \nonumber \\
&+\int_{Y'}^Z-\fr{1+\cos w}{4c}p(X',Y')\ {\rm d}X',\ \ (Z=X,Y),
\end{align}
\begin{align}\label{2.33}
x_n(X',Y)=x(X',\phi(X'))+\int_{\phi(X')}^{Y}-\fr{1+\cos z}{4}q(X',Y')\ {\rm d}Y',
\end{align}
and
\begin{align}\label{2.34}
t_n(X',Y)=t(X',\phi(X'))+\int_{\phi(X')}^{Y}\fr{1+\cos z}{4c}q(X',Y')\ {\rm d}Y'.
\end{align}
It is noticed that $(x_m,t_m)(Z,Y')$ and $(x_n,t_n)(X',Y)$ have finite partial derivatives in $Z$ and $Y$, respectively.

We denote
$$
V=(\theta, w, z, p, q),\quad \hat{V}=(\hat{\theta}, \hat{w}, \hat{z}, \hat{p}, \hat{q}),
$$
and
$$
V_0(X)=(\theta, w, z, p, q)(X, 0),\ \ \forall\ X\in[0,\delta_2].
$$
Note that the functions $(\theta, w, z, p, q)(X, 0)$ are achieved by solving the problem in $\Omega_{2\delta_2}^0$. Due to the result in \cite{CHL20}, one knows that the vector function $V_0(X)$ is $C^\alpha$ continuous.
It is obvious that
\begin{align}\label{2.35}
\hat{\theta}(X,0)=\theta(X,0),\ \ \hat{w}(X,0)=w(X,0),\ \ \hat{p}(X,0)=p(X,0),
\end{align}
and
\begin{align}\label{2.36}
\hat{z}(X,0)=&-w(0,0)+ \int_{0}^{X}\fr{p}{4c}\bigg\{\fr{c'}{c}\bigg(\cos^2\fr{w}{2}-\cos^2\fr{z}{2}\bigg) \nonumber \\
& -\sin w\cos^2\fr{z}{2}-\sin z\cos^2\fr{w}{2} -4J(x_n,t_n)\cos^2\fr{w}{2}\cos^2\fr{z}{2}\bigg\} (X',0)\ {\rm d}X' \nonumber \\
=& z(0,0)+ \int_{0}^{X}\fr{p}{4c}\bigg\{\fr{c'}{c}\bigg(\cos^2\fr{w}{2}-\cos^2\fr{z}{2}\bigg) -\sin w\cos^2\fr{z}{2}\nonumber \\
&-\sin z\cos^2\fr{w}{2} -4J(x_n,t_n)\cos^2\fr{w}{2}\cos^2\fr{z}{2}\bigg\} (X',0)\ {\rm d}X'=z(X,0).
\end{align}
The relation $\hat{q}(X,0)=q(X,0)$ can be checked similarly by \eqref{2.30}.

Let $K_1$ be a sufficiently large positive constant that only depends on $\|V_0(X)\|_{C^\alpha}$. Set
\begin{align}\label{2.37}
\mathcal{K}_1=\bigg\{V\ \big|\ &\|V(X,Y)\|_{C^\alpha(\Omega_{1\delta_1}^0)}\leq K_1,\ \ (w+z)(X,X)=0, \nonumber \\
& (p-q)(X,X)=0,\ \ V(X,0)=V_0(X)\bigg\},
\end{align}
where $\delta_1\leq\delta_2$ is a small positive constant and
\begin{align}\label{2.38}
\Omega_{1\delta_1}^0=\big\{(X,Y)\in\Omega_{1\delta_2}^0:\ {\rm dist}((X,Y),(0,0))\leq \delta_1\big\}.
\end{align}
It is easy to see that $\hat{V}(X,0)=V_0(X)$ and $\mathcal{K}_1$ is a compact set in $C^0(\Omega_{1\delta_1}^0)$ space. To apply the Schauder fixed point theorem, we only need to show that the map $\mathcal{T}_1$ is continuous under $C^0$ norm and maps $\mathcal{K}_1$ to itself. By selecting appropriate constants $K_1$ and then $\delta_1$, these properties can be checked based on the facts that $J(x,t)$ is $C^\alpha$ continuous in $(x,t)$ and $(x_m, t_m)(Z,Y')$ and $(x_n,t_n)(X',Y)$ have finite partial derivatives with respect to $Z, Y$. We omit the proof since it is entirely similar to that in \cite{CHL20}.

\subsection{Inverse transformation}\label{S2.4}

We recall \eqref{2.17} and \eqref{2.18} to calculate
\begin{align}\label{2.39}
\bigg(\fr{1+\cos w}{4}p\bigg)_Y=\bigg(-\fr{1+\cos z}{4}q\bigg)_X, \quad \bigg(\fr{1+\cos w}{4c}p\bigg)_Y=\bigg(\fr{1+\cos z}{4c}q\bigg)_X,
\end{align}
which mean that
\begin{align}\label{2.40}
x_{XY}=x_{YX},\quad t_{XY}=t_{YX}.
\end{align}
Thus two $x$ equations and two $t$ equations in \eqref{2.18} are equivalent, respectively. Therefore, we have by \eqref{2.31}-\eqref{2.34}
$$
(x_m(X,Y),t_m(X,Y))=(x_n(X,Y),t_n(X,Y))=:(x,t).
$$
Moreover, if the solution exists, we can define the functions $(x(X,Y),t(X,Y))$ based on the region where $(X,Y)$ is located. For example, if $(X,Y)\in\Omega_{1}^2$, then
\begin{align}\label{2.41}
x(X,Y)=\int_{Y}^X\fr{1+\cos w}{4}p(X',Y)\ {\rm d}X' =\pi+\int_{Y+\bar{X}}^Y-\fr{1+\cos z}{4}q(X,Y')\ {\rm d}Y',
\end{align}
and
\begin{align}\label{2.42}
t(X,Y)=&\int_{\phi(Y-\widetilde{X})}^{Y-\widetilde{X}}\fr{1+\cos z}{4c}q(Y-\bar{X}, Y')\ {\rm d}Y' +\int_{Y-\widetilde{X}}^{Y}\fr{1+\cos w}{4c}p(X', Y-\widetilde{X})\ {\rm d}X' \nonumber \\
&+\int_{Y-\widetilde{X}}^{Y}\fr{1+\cos z}{4c}q(Y, Y')\ {\rm d}Y' +\int_{Y}^X\fr{1+\cos w}{4c}p(X', Y)\ {\rm d}X'  \nonumber \\
=& \int_{\phi(X-\widetilde{X})}^{X-\widetilde{X}}\fr{1+\cos z}{4c}q(X-\widetilde{X}, Y')\ {\rm d}Y' +\int_{X-\widetilde{X}}^{X}\fr{1+\cos w}{4c}p(X', X-\widetilde{X})\ {\rm d}X' \nonumber \\
&+\int_{X-\widetilde{X}}^Y\fr{1+\cos z}{4c}q(X, Y')\ {\rm d}Y'.
\end{align}
We point out that the map from $(X,Y)$ to $(x,t)$ constructed above may not be one-to-one
mapping. But the values of $\theta$ do not depend on the choice of $(X, Y)$, that is, if $x(X_1,Y_1)=x(X_2,Y_2)$ and $t(X_1,Y_1)=t(X_2,Y_2)$ for two points $(X_1,Y_1)$ and $(X_2,Y_2)$ in $\widetilde{\Omega}$, we have
$$
\theta(x(X_1,Y_1),t(X_1,Y_1))=\theta(x(X_2,Y_2),t(X_2,Y_2)).
$$
To show this assertion, we divide the proof into two cases: Case 1. $X_1\leq X_2, Y_1\leq Y_2$ and Case 2. $X_1\leq X_2, Y_1\geq Y_2$. The proof for Case 2 is identical to that in Bressan and Zheng \cite{Bressan-Zheng}. For Case 1, if $x(X_1,Y_1)=x(X_2,Y_2)=x^*\in(0,\pi)$, as in \cite{Bressan-Zheng}, one considers the set
$$
D_{X^*}:=\{(X,Y):\ x(X,Y)\leq x^*\}
$$
and denote by $\partial D_{X^*}$ its boundary. Due to the facts that $x$ is increasing with $X$ and decreasing with $Y$, this boundary is a Lipschitz continuous curve in $\Omega$. Hence we can
construct a Lipschitz continuous curve $\gamma_1$ connecting points $(X_1, Y_1)$ and $(X_2, Y_2)$, which
consists a horizontal segment $Y\equiv Y_1$, $\partial D_{X^*}$ and a vertical segment $X\equiv X_2$. On $\gamma_1$, there hold $x(X,Y)\equiv x(X_1,Y_1)$ and $t(X,Y)\equiv t(X_1,Y_1)$ by \eqref{2.18}. Thus we find that
$$
\fr{1+\cos w}{4}p\ {\rm d}X=\fr{1+\cos z}{4}q\ {\rm d}Y=0,
$$
along $\gamma_1$, which imply by \eqref{2.17} that
$$
\theta(x(X_2,Y_2),t(X_2,Y_2))-\theta(x(X_1,Y_1),t(X_1,Y_1))=\int_{\gamma_1}\fr{\sin w}{4c}p\ {\rm d}X +\fr{\sin z}{4c}q\ {\rm d}Y=0.
$$
If $x(X_1,Y_1)=x(X_2,Y_2)=0$ or $\pi$, we then use the line $L_0$ or $L_\pi$ to replace the boundary $\partial D_{X^*}$ as before. For instance, if $x(X_1,Y_1)=x(X_2,Y_2)=\pi$, a Lipschitz continuous curve $\gamma_2$ connecting points $(X_1, Y_1)$ and $(X_2, Y_2)$ can be constructed by a horizontal segment $Y\equiv Y_1$, $L_\pi$ and a vertical segment $X\equiv X_2$. Due to $t(X_1, Y_1)=t(X_2, Y_2)$, one has by \eqref{2.18} and the boundary condition $(1+\cos w)p=(1+\cos z)q$ on $L_\pi$
\begin{align*}
0=&t(X_2, Y_2)-t(X_1, Y_1)=t(X_2,X_2-\widetilde{X})-t(Y_1+\widetilde{X},Y_1) \\ =&\int_{(Y_1+\widetilde{X},Y_1)}^{(X_2,X_2-\widetilde{X})} \fr{(1+\cos w)p}{4c}\ {\rm d}X +\fr{(1+\cos z)q}{4c}\ {\rm d}Y \\
=& \int_{Y_1+\widetilde{X}}^{X_2} \fr{(1+\cos w)p}{2c}(X,X-\widetilde{X})\ {\rm d}X,
\end{align*}
which implies that $(1+\cos w)p(X,X-\widetilde{X})=(1+\cos z)q(X,X-\widetilde{X})=0$ for $X\in[Y_1+\widetilde{X}, X_2]$. Thus we utilize the equations in \eqref{2.17} to get along $\gamma_2$
\begin{align*}
&\theta(x(X_2,Y_2),t(X_2,Y_2))-\theta(x(X_1,Y_1),t(X_1,Y_1)) \\ =&\theta(\pi,t(X_2,X_2-\widetilde{X}))-\theta(\pi,t(Y_1+\widetilde{X},Y_1))  \\
=& \int_{(Y_1+\widetilde{X},Y_1)}^{(X_2,X_2-\widetilde{X})} \fr{\sin w p}{4c}\ {\rm d}X +\fr{\sin z q}{4c}\ {\rm d}Y=0.
\end{align*}

In addition, one also has
\begin{align}\label{2.73a}
{\rm d}x{\rm d}t=\fr{pq}{2c(1+R^2)(1+S^2)}{\rm d}X{\rm d}Y=\fr{pq}{2c}\cos^2\fr{w}{2}\cos^2\fr{z}{2}{\rm d}X{\rm d}Y.
\end{align}

We now show that the estimate
\begin{align}\label{2.73}
E(t)=\int_0^\pi (\theta_{t}^2+c^2(\theta)\theta_{x}^2)(x,t)\ {\rm d}x +2B(\theta(\pi,t))\leq C_E,
\end{align}
for $t\in[0,\widetilde{T}]$, where $C_E$ is a positive constant depending on $E(0)$ and $J$, $\widetilde{T}$ is the existence time of solution, and $B(\theta(\pi,t))$ is the boundary energy defined in \eqref{1.13}. Let $\Gamma_t\subset\widetilde{\Omega}_{\widetilde{T}}$ be the transformation of the horizontal segment of $t$ with $x\in[0,\pi]$ in the $(x,t)$ plane, where $\widetilde{\Omega}_{\widetilde{T}}$ is the corresponding region of $\Omega_{\widetilde{T}}=[0,\pi]\times[0,\widetilde{T}]$ in the $(X,Y)$ plane. We use $(X_0, X_0)$ and $(X_\pi , X_\pi -\widetilde{X})$ to represent the  coordinates of the intersection points of $\Gamma_t$ and the lines $L_0$, $L_\pi$, respectively. Denote $\Omega_t=[0,\pi]\times[0,t]$ and $\widetilde{\Omega}_t$ the corresponding region of $\Omega_t$ in the $(X,Y)$ plane.
Then one has
\begin{align}\label{2.74}
&\int_0^\pi (\theta_{t}^2+c^2(\theta)\theta_{x}^2)(x,t)\ {\rm d}x =\int_{\Gamma_t\cap\{\cos w\neq-1\}}\fr{1-\cos w}{4}p\ {\rm d}X-\int_{\Gamma_t\cap\{\cos z\neq-1\}}\fr{1-\cos z}{4}q\ {\rm d}Y \nonumber \\
\leq & \int_{\Gamma_t}\fr{1-\cos w}{4}p\ {\rm d}X-\fr{1-\cos z}{4}q\ {\rm d}Y \nonumber \\
=& \bigg\{
\int_{\Gamma_0} + \int_{(\widehat{X},\phi(\widehat{X}))}^{(X_\pi ,X_\pi -\widetilde{X})} - \int_{(0,0)}^{(X_0,X_0)}\bigg\}\fr{1-\cos w}{4}p\ {\rm d}X-\fr{1-\cos z}{4}q\ {\rm d}Y   \nonumber \\
&  -\iint_{\widetilde{\Omega}_t}\fr{pq}{4c}\bigg(\sin\fr{w}{2}\cos\fr{z}{2}+\sin\fr{z}{2}\cos\fr{w}{2}\bigg)^2\ {\rm d}X{\rm d}Y \nonumber \\
& -\iint_{\widetilde{\Omega}_t}\fr{pq}{4c}J\bigg(\sin w\cos^2\fr{z}{2}+\sin z\cos^2\fr{w}{2}\bigg) \ {\rm d}X{\rm d}Y \nonumber \\
= &\int_{\Gamma_0} \fr{1-\cos w}{4}p\ {\rm d}X-\fr{1-\cos z}{4}q\ {\rm d}Y + \int_{(\widehat{X},\phi(\widehat{X}))}^{(X_\pi ,X_\pi -\widetilde{X})} \fr{1-\cos w}{4}p\ {\rm d}X-\fr{1-\cos z}{4}q\ {\rm d}Y \nonumber \\ & -\iint_{\widetilde{\Omega}_t}\fr{pq}{4c}\cos^2\fr{w}{2}\cos^2\fr{z}{2} \bigg(\tan\fr{w}{2}+\tan\fr{z}{2}\bigg)^2\ {\rm d}X{\rm d}Y \nonumber \\
& -\iint_{\widetilde{\Omega}_t}\fr{pq}{2c}\cos^2\fr{w}{2}\cos^2\fr{z}{2} 2J\bigg(\tan\fr{w}{2}+\tan\fr{z}{2}\bigg) \ {\rm d}X{\rm d}Y.
\end{align}
Here we have used the boundary conditions $w+z=0, p=q$ on $L_0$.
Moreover, recalling the boundary conditions \eqref{2.23} on $L_\pi$ arrives at
\begin{align}\label{2.74a}
&\int_{(\widehat{X},\phi(\widehat{X}))}^{(X_\pi ,X_\pi -\widetilde{X})} \fr{1-\cos w}{4}p\ {\rm d}X-\fr{1-\cos z}{4}q\ {\rm d}Y \nonumber \\
=&\int_{\phi(\widehat{X})}^{X_\pi -\widetilde{X}} \bigg(\fr{2\sin^2\fr{w}{2}}{4}\fr{1+\tan^2\fr{w}{2}}{1+\tan^2\fr{z}{2}}-\fr{2\sin^2\fr{z}{2}}{4} \bigg)q(Y+\widetilde{X},Y)\ {\rm d}Y \nonumber \\
=&\int_{\phi(\widehat{X})}^{X_\pi -\widetilde{X}} \bigg(\tan^2\fr{w}{2}-\tan^2\fr{z}{2} \bigg)\fr{\cos^2\fr{z}{2}}{2}q(Y+\widetilde{X},Y)\ {\rm d}Y.
\end{align}
Along $L_\pi$, we have by \eqref{2.18} and the boundary conditions \eqref{2.23}
\begin{align*}
{\rm d}t=\fr{1+\cos w}{4c}p\ {\rm d}X+\fr{1+\cos z}{4c}q\ {\rm d}Y=\fr{\cos^2\fr{z}{2}}{c}q\ {\rm d}Y.
\end{align*}
Putting the above into \eqref{2.74a} and using the boundary condition $\theta_x=-\iota \theta$ on $x=\pi$ yields
\begin{align}\label{2.74b}
&\int_{(\widehat{X},\phi(\widehat{X}))}^{(X_\pi ,X_\pi -\widetilde{X})} \fr{1-\cos w}{4}p\ {\rm d}X-\fr{1-\cos z}{4}q\ {\rm d}Y \nonumber \\
=&\int_{0}^t -2\iota c^2(\theta)\theta\theta_t(\pi,t)\ {\rm d}t=-2\int_{\theta(\pi,0)}^{\theta(\pi,t)} \iota c^2(s)s\ {\rm d}s=2B(\theta(\pi,0))-2B(\theta(\pi,t)).
\end{align}
We now insert \eqref{2.74b} into \eqref{2.74} and utilize \eqref{2.73a} to obtain
\begin{align}\label{2.74c}
\int_0^\pi (\theta_{t}^2+c^2(\theta)\theta_{x}^2)(x,t)\ {\rm d}x \leq & \int_0^\pi  (\theta_{t}^2+c^2(\theta)\theta_{x}^2)(x,0)\ {\rm d}x +2B(\theta(\pi,0))
\nonumber \\
& -2B(\theta(\pi,t)) -2\iint_{\Omega_t}\theta_{t}^2\ {\rm d}x{\rm d}t -2\iint_{\Omega_t}J\theta_{t}\ {\rm d}x{\rm d}t,
\end{align}
from which we get
\begin{align}\label{2.75}
E(t)\leq &E(0) -2\iint_{\Omega_t}\theta_{t}^2\ {\rm d}x{\rm d}t -\iint_{\Omega_t}2|J|\cdot|\theta_{t}|\ {\rm d}x{\rm d}t \nonumber \\
\leq &E(0) +\iint_{\Omega_t}|J|^2\ {\rm d}x{\rm d}t\leq E(0)+\bar{J}(T)\pi T,
\end{align}
which means that
\begin{align}\label{2.76}
\max_{0\leq t\leq \widetilde{T}}E(t)\leq E(0)+\bar{J}(T)\pi T.
\end{align}
Furthermore, the inequality in \eqref{2.76} indicates that $(\theta_t, \theta_x)(\cdot,t)$ and then $(R, S)(\cdot,t)$ are square integrable functions in $x$.

\subsection{Global existence of the boundary value problem to the semilinear system}\label{S2.5}

Let $T>0$ be an any fixed time. Denote
$$
\Omega_T=\{(x,t)\ x\in[0,\pi],\ t\in[0,T]\}.
$$
We use $\widetilde{\Omega}_T$ to represent the image of $\Omega_T$ in the $(X,Y)$ plane. For any given $J(x,t)$ satisfying
\begin{align}\label{2.43}
\|J\|_{C^\alpha\cap L^\infty(\Omega_T)}=: \bar{J}(T)<\infty,
\end{align}
we next extend the local solution constructed in Subsection \ref{S2.3} to the whole region $\widetilde{\Omega}_T$. It suffices to establish the global a priori estimates on $p$ and $q$.

\begin{lem}
Let $(\theta, w, z, p, q)(X,Y)$ be a solution of the boundary value problem \eqref{2.17}, \eqref{2.23} on $\widetilde{\Omega}_T$. Then there exist two positive constants $M$ and $N$ depending only on the boundary data on $\Gamma_0$ and $\bar{J}(T)$ such that
\begin{align}\label{2.44}
0<M\leq \max_{(X,Y)\in\widetilde{\Omega}_T}\{p(X,Y), q(X,Y)\}\leq N.
\end{align}
\end{lem}
\begin{proof}
For convenience, we write equations of $p$ and $q$ in \eqref{2.17} as
\begin{align}\label{2.45}
p_Y=Fpq,\quad q_X=Gpq.
\end{align}
We first discuss the region $\Omega^0$. From \eqref{2.45}, we know by $p=q=1$ on $\Gamma_0$ and the boundary conditions on $L_0, L_\pi$ that $p$ and $q$ are positive on $\Omega^0$.

For any point $(X,Y)$ in $\Omega_{1}^0$, we consider the region $\Sigma_{1}^0$ enclosed by a horizontal segment between $(Y,Y)$ and $(X,Y)$,
a vertical segment between $(X,Y)$ and $(X,\phi(X))$, $\Gamma_0$ and $L_0$. By a direct calculate, one obtains
\begin{align}\label{2.46}
\int_{\pa \Sigma_{1}^0}p\ {\rm d}X'-q\ {\rm d}Y'=&-\iint_{\Sigma_{1}^0}q_X+p_Y\ {\rm d}X'{\rm d}Y' \nonumber \\
=&\iint_{\Sigma_{1}^0}\fr{pq}{2c}\bigg(\sin\fr{w}{2}\cos\fr{z}{2} +\sin\fr{z}{2}\cos\fr{w}{2}\bigg)^2\ {\rm d}X'{\rm d}Y' \nonumber \\
&+\iint_{\Sigma_{1}^0}\fr{pq}{2c}J\bigg(\sin z\cos^2\fr{w}{2} +\sin w\cos^2\fr{z}{2}\bigg)\ {\rm d}X'{\rm d}Y'.
\end{align}
In view of the construction of $\Sigma_{1}^0$, we apply the facts $p=q=1$ on $\Gamma_0$ and $p=q$ on $L_0$ again
\begin{align}\label{2.47}
\int_{\pa \Sigma_{1}^0}p\ {\rm d}X'-q\ {\rm d}Y'=&-\int_{Y}^X p(X',Y)\ {\rm d}X' -\int_{\phi(X)}^Y q(X,Y')\ {\rm d}Y' \nonumber \\
&+\int_{0}^X 1\ {\rm d}X'-1\ {\rm d}\phi(X') -\int_{(0,0)}^{(Y,Y)}p\ {\rm d}X'-q\ {\rm d}X' \nonumber \\
=&X-\phi(X)-\int_{Y}^X p(X',Y)\ {\rm d}X' -\int_{\phi(X)}^Y q(X,Y')\ {\rm d}Y'.
\end{align}
Putting \eqref{2.47} into \eqref{2.46} yields
\begin{align}\label{2.48}
&\int_{Y}^X p(X',Y)\ {\rm d}X' +\int_{\phi(X)}^Y q(X,Y')\ {\rm d}Y' \nonumber \\
=&X-\phi(X) -\iint_{\Sigma_{1}^0}\fr{pq}{2c}\bigg(\sin\fr{w}{2}\cos\fr{z}{2} +\sin\fr{z}{2}\cos\fr{w}{2}\bigg)^2\ {\rm d}X'{\rm d}Y'  \nonumber \\
&-\iint_{\Sigma_{1}^0}\fr{pq}{2c}J\bigg(\sin z\cos^2\fr{w}{2} +\sin w\cos^2\fr{z}{2}\bigg)\ {\rm d}X'{\rm d}Y' \nonumber \\
= &X-\phi(X) -\iint_{\Sigma_{1}^0}\fr{pq}{2c}\bigg(\tan\fr{w}{2}+\tan\fr{z}{2}\bigg)^2\cos^2\fr{w}{2}\cos^2\fr{z}{2}\ {\rm d}X'{\rm d}Y'  \nonumber \\
&-\iint_{\Sigma_{1}^0}\fr{pq}{2c}2J\bigg(\tan\fr{w}{2}+\tan\fr{z}{2}\bigg)\cos^2\fr{w}{2}\cos^2\fr{z}{2}\ {\rm d}X'{\rm d}Y' \nonumber \\
\leq & X-\phi(X) +\iint_{\Sigma_{1}^0}\fr{pq}{2c}|J|^2\cos^2\fr{w}{2}\cos^2\fr{z}{2}\ {\rm d}X'{\rm d}Y'
\nonumber \\
\leq & \widehat{X}+\widetilde{X} +\iint_{\tilde\Sigma_{1}^0}|J|^2\ {\rm d}x{\rm d}t \leq \widehat{X}+\widetilde{X} +\bar{J}(T)\pi T.
\end{align}
Here $\tilde\Sigma_{1}^0$ is the region in the $(x,t)$ plane transformed from $\Sigma_{1}^0$. Moreover, if the point $(X,Y)$ on $L_0$, that is $X=Y$, we directly have by \eqref{2.48}
\begin{align}\label{2.49}
\int_{\phi(Y)}^Y q(Y,Y')\ {\rm d}Y'\leq \widehat{X}+\widetilde{X} +\bar{J}(T)\pi T.
\end{align}
Now we integrate \eqref{2.45} and employ \eqref{2.48}, \eqref{2.49} to acquire
\begin{align}\label{2.50}
p(X,Y)=&p(X,\phi(X))\exp\bigg(\int_{\phi(X)}^YFq(X,Y')\ {\rm d}Y'\bigg) \nonumber \\
\leq &\exp\bigg(\tilde{F}\int_{\phi(X)}^Yq(X,Y')\ {\rm d}Y'\bigg)\leq \exp\bigg(\tilde{F} [\widehat{X}+\widetilde{X} +\bar{J}(T)\pi T]\bigg),
\end{align}
and
\begin{align}\label{2.51}
q(X,Y)=&q(Y,Y)\exp\bigg(\int_{Y}^XGp(X',Y)\ {\rm d}X'\bigg) =p(Y,Y)\exp\bigg(\int_{Y}^XGp(X',Y)\ {\rm d}X'\bigg) \nonumber \\
=&p(Y,\phi(Y))\exp\bigg(\int_{\phi(Y)}^YFq(Y,Y')\ {\rm d}Y'\bigg)\cdot\exp\bigg(\int_{Y}^XGp(X',Y)\ {\rm d}X'\bigg) \nonumber \\
\leq& \exp\bigg(\tilde{F}\int_{\phi(Y)}^Yq(Y,Y')\ {\rm d}Y'\bigg)\cdot\exp\bigg(\tilde{F}\int_{Y}^Xp(X',Y)\ {\rm d}X'\bigg)
\nonumber \\
\leq& \exp\bigg(2\tilde{F} [\widehat{X}+\widetilde{X} +\bar{J}(T)\pi T]\bigg),
\end{align}
where
\begin{align*}
|F|, |G|\leq \tilde{F}=\fr{1}{2C_L}\bigg(\fr{C_1}{2C_L}+2+\bar{J}(T)\bigg).
\end{align*}
Here the assumption \eqref{1.6} is applied.

The case $(X,Y)\in\Omega_{2}^0$ can be easily handled. We now discuss the case $(X,Y)\in\Omega_{3}^0$.
Consider the region $\Sigma_3^0$ enclosed by a vertical segment between $(X,Y)$ and $(X, X-\widetilde X)$, a horizontal segment between $(X,Y)$ and $(\phi^{-1}(Y), Y)$, $\Gamma_0$ and $L_\pi$. By Green's theorem, one also has
\begin{align}\label{2.46a}
\int_{\pa \Sigma_{3}^0}p\ {\rm d}X'-q\ {\rm d}Y'=&-\iint_{\Sigma_{3}^0}q_X+p_Y\ {\rm d}X'{\rm d}Y' \nonumber \\
=&\iint_{\Sigma_{3}^0}\fr{pq}{2c}\cos^2\fr{w}{2}\cos^2\fr{z}{2}\bigg(\tan\fr{w}{2}+\tan\fr{z}{2}\bigg)^2\ {\rm d}X'{\rm d}Y' \nonumber \\
&+\iint_{\Sigma_{3}^0}\fr{pq}{2c}\cos^2\fr{w}{2}\cos^2\fr{z}{2}2J\bigg(\tan\fr{w}{2}+\tan\fr{z}{2}\bigg)\ {\rm d}X'{\rm d}Y'.
\end{align}
Due to the construction of $\Sigma_{3}^0$, we employ the boundary conditions in \eqref{2.23} on $\Gamma_0$ to get
\begin{align}\label{2.47a}
\int_{\pa \Sigma_{3}^0}p\ {\rm d}X'-q\ {\rm d}Y'=&\int_{X-\widetilde X}^Y -q(X,Y')\ {\rm d}Y' -\int_{\phi^{-1}(Y)}^X p(X',Y)\ {\rm d}X' \nonumber \\
&\  -\int_{\widehat X}^{\phi^{-1}(Y)} 1\ {\rm d}X'-1\ {\rm d}\phi(X') +\int_{(\widehat X,\widehat X-\widetilde X)}^{(X, X-\widetilde{X})}p\ {\rm d}X'-q\ {\rm d}X'  \nonumber \\
=&\widehat X-\phi^{-1}(Y)+Y-\phi(\widehat X)-\int_{X-\widetilde X}^Y q(X,Y')\ {\rm d}Y' \nonumber \\
& \ -\int_{\phi^{-1}(Y)}^X p(X',Y)\ {\rm d}X'  +\int_{\widehat X-\widetilde X}^{X-\widetilde{X}}(p-q)(Y'+\widetilde{X},Y')\ {\rm d}Y'.
\end{align}
We next estimate the last term in \eqref{2.47a}. Making use of the boundary conditions in \eqref{2.23} on $L_\pi$ leads to
\begin{align}\label{2.47b}
& \int_{\widehat X-\widetilde X}^{X-\widetilde{X}}(p-q)(Y'+\widetilde{X},Y')\ {\rm d}Y' =\int_{\widehat X-\widetilde X}^{X-\widetilde{X}}q\left[\frac{1+(\tan\frac z2-2\iota c\theta)^2}{1+\tan^2\frac z2}-1\right](Y'+\widetilde{X},Y')\ {\rm d}Y' \nonumber \\
=& \int_{\widehat X-\widetilde X}^{X-\widetilde X} \left[\frac{1+(\tan\frac z2-2\iota c\theta)^2}{1+\tan^2\frac z2}-1\right]_{\Gamma_0} {\rm d}Y'+\iint_{\Sigma_{3,1}^0}\partial_X\left[q\left(\frac{(\tan\frac z2-2\iota c\theta)^2-\tan^2\frac z2}{1+\tan^2\frac z2}\right)\right]\ {\rm d}X'{\rm d}Y',
\end{align}
where $\Sigma_{3,1}^0$ is the region enclosed by a horizontal segment between $(X,X-\widetilde{X})$ and $(\phi^{-1}(X-\widetilde{X}), X-\widetilde{X})$, $\Gamma_0$ and $L_\pi$. Performing direct calculations, one obtains by system \eqref{2.17}
\begin{align}\label{2.47c}
&\partial_X\left[q\left(\frac{(\tan\frac z2-2\iota c\theta)^2-\tan^2\frac z2}{1+\tan^2\frac z2}\right)\right]
=\partial_X\left[q(-2\iota c\theta\sin z+2\iota^2c^2\theta^2\cos z+2\iota^2c^2\theta^2)\right] \nonumber \\
=& q_X\left[-2\iota c\theta\sin z+2\iota^2c^2\theta^2\cos z+2\iota^2c^2\theta^2\right] +z_X\left[-2\iota q c\theta\cos z-2\iota^2qc^2\theta^2\sin z\right] \nonumber \\
&+\theta_X\left[-2\iota (c'\theta+c)q\sin z+2\iota^2q(2cc'\theta^2+2c^2\theta)(\cos z+1)\right]
\nonumber \\
=&\frac{pq}{2c}\cos^2\frac w2\cos^2\frac z2(\iota \Lambda_1+\iota^2\Lambda_2),
\end{align}
where
\begin{align*}
\Lambda_1=~&c\theta\left[4\tan\frac w2\sin^2 \frac z2+2\tan \frac z2-2\tan \frac w2\right]-4c\tan \frac z2 \tan \frac w2  \\
&+c'\theta\bigg(2+\tan^2\frac w2+\tan^2\frac z2-6\tan \frac z2\tan \frac w2\bigg)   -2c\theta\sin z  \\
&+8c\theta\sin^2 \frac z2 J +4c\theta J\cos z+2c\theta\sin z+2c\theta\sin w-2c'\theta,
\end{align*}
and
\begin{align*}
\Lambda_2
=~&\bigg[\fr{c'}{2c}\tan \frac w2-\fr{c'}{4c}\sin z(1+\tan^2\frac w2)-\fr{1}{2}\tan \frac w2\sin z -\sin^2\fr{z}{2} -J\sin z\bigg]4c^2\theta^2\\
&+8\tan \frac w2\left[(cc'\theta^2+c^2\theta)\right]\\
&+\bigg[\fr{c'}{c}\bigg(\sin\fr{z}{2}\tan\fr{z}{2}-\cos\fr{z}{2}\tan^2\fr{w}{2}\bigg)-2\tan \frac w2 -2\sin \frac z2 -4J\cos\fr{z}{2}\bigg]\left[-2c^2\theta^2\sin \frac z2\right].
\end{align*}
Note that the quadratic terms of $\tan\fr{w}{2}$ and $\tan\fr{z}{2}$ appear at most in $\Lambda_1, \Lambda_2$, and $\Sigma_{3,1}^0\subset\Omega^0$ is a bounded region. We use \eqref{2.73a} and the estimate \eqref{2.76} to acquire
\begin{align}\label{2.47d}
&\left|\iint_{\Sigma_{3,1}^0}\partial_X\left[q\left(\frac{(\tan\frac z2-2\iota c\theta)^2-\tan^2\frac z2}{1+\tan^2\frac z2}\right)\right]\ {\rm d}X'{\rm d}Y'\right| \nonumber \\
= &\left|\iint_{\Sigma_{3,1}^0}\frac{pq}{2c}\cos^2\frac w2\cos^2\frac z2(\iota \Lambda_1+\iota^2\Lambda_2)\ {\rm d}X'{\rm d}Y'\right|\les (\iota+\iota^2)[E(0)+1+\bar{J}(T)]T.
\end{align}
Here and below, we use $A\les B$ to denote $A\leq CB$ for some uniform constant $C$.
Combining \eqref{2.46a}, \eqref{2.47a}, \eqref{2.47b} and \eqref{2.47d}, one achieves
\begin{align}\label{2.47e}
& \int_{X-\tilde X}^Y q(X,Y')\ {\rm d}Y' +\int_{\phi^{-1}(Y)}^X p(X',Y)\ {\rm d}X' \nonumber \\
\leq\ & \widehat X-\phi^{-1}(Y)+Y-\phi(\widehat X)+\int_{(\widehat X,\widehat X-\widetilde X)}^{(X, X-\widetilde{X})}(p-q)(Y'+\widetilde{X},Y')\ {\rm d}Y'  \nonumber \\
& -\iint_{\Sigma_{3}^0}\fr{pq}{2c}\cos^2\fr{w}{2}\cos^2\fr{z}{2}\bigg(\tan\fr{w}{2}+\tan\fr{z}{2}\bigg)^2\ {\rm d}X'{\rm d}Y' \nonumber \\
& -\iint_{\Sigma_{3}^0}\fr{pq}{2c}\cos^2\fr{w}{2}\cos^2\fr{z}{2}2J\bigg(\tan\fr{w}{2}+\tan\fr{z}{2}\bigg)\ {\rm d}X'{\rm d}Y' \nonumber \\
\lesssim\ & \widehat X+\widetilde X+\bar J(T)\pi T+(\iota+\iota^2)[E(0)+1+\bar{J}(T)]T.
\end{align}
With similar arguments as \eqref{2.50} and \eqref{2.51}, we can obtain the upper bounds of $p(X,Y)$ and $q(X,Y)$ for all $(X,Y)\in\Omega_{3}^0$.
Thus there exist two positive constants $M_0$ and $N_0$ depending only on the boundary data on $\Gamma_0$ and $\bar{J}(T)$ such that
\begin{align}\label{2.52}
0<M_0\leq \max_{(X,Y)\in\Omega^0}\{p(X,Y), q(X,Y)\}\leq N_0.
\end{align}

Thanks to \eqref{2.52}, there hold
\begin{align}\label{2.53}
0<M_0\leq \max_{(X,Y)\in\Gamma_1}\{p(X,Y), q(X,Y)\}\leq N_0,
\end{align}
where $\Gamma_1=\Gamma_{11}\cup\Gamma_{12}$,
$$
\Gamma_{11}=\{(X,Y):\ X=\widehat{X},\ 0\leq Y\leq \widehat{X}\},\quad \Gamma_{12}=\{(X,Y):\ \widehat{X}\leq X\leq \widetilde{X},\ Y=0\},
$$
which are two boundaries of the region $\Omega^1$. For any point $(X,Y)$ in $\Omega^1$, for example, $(X,Y)\in\Omega_{3}^1$, we consider the region $\Sigma_{3}^1$ enclosed by a horizontal segment between $(\widehat{X},Y)$ and $(X,Y)$,
a vertical segment between $(X,Y)$ and $(X,X-\widetilde{X})$, $L_\pi$, $\Gamma_{12}$ and $\Gamma_{11}$. According to the construction of $\Sigma_{3}^1$, one has as in \eqref{2.47}
\begin{align}\label{2.54}
&\int_{\pa \Sigma_{3}^1}p\ {\rm d}X'-q\ {\rm d}Y'=-\int_{\widehat{X}}^X p(X',Y)\ {\rm d}X' -\int_{X-\widetilde{X}}^Y q(X,Y')\ {\rm d}Y' \nonumber \\
&\qquad +\int_{0}^Y q(\widehat{X},Y')\ {\rm d}Y' +\int_{\widehat{X}}^{\widetilde{X}} p(X',0)\ {\rm d}X' +\int_{(\widetilde{X},0)}^{(X,X-\widetilde{X})}p\ {\rm d}X'-q\ {\rm d}(X'-\widetilde{X}) \nonumber \\
\leq &N_0(Y+\widetilde{X}-\widehat{X})-\int_{\widehat{X}}^X p(X',Y)\ {\rm d}X' -\int_{X-\widetilde{X}}^Y q(X,Y')\ {\rm d}Y' \nonumber \\
&+\int_{0}^{X-\widetilde{X}}(p-q)(Y'+\widetilde{X},Y')\ {\rm d}Y'.
\end{align}
The last term in \eqref{2.54} can be treated as in \eqref{2.47b}. Hence similar to \eqref{2.47e}, we can acquire
\begin{align}\label{2.55}
&\int_{\widehat{X}}^X p(X',Y)\ {\rm d}X' +\int_{X-\widetilde{X}}^Y q(X,Y')\ {\rm d}Y'\nonumber \\ \lesssim & N_0(\widetilde{X}+\widehat{X}) +\bar J(T)\pi T+(\iota+\iota^2)[E(0)+1+\bar{J}(T)]T=:C_0(\iota, \bar{J}, T, N_0).
\end{align}
Integrating \eqref{2.45} and utilizing \eqref{2.55} and the boundary conditions \eqref{2.23} give
\begin{align}\label{2.56}
p(X,Y)=&p(X,X-\widetilde{X})\exp\bigg(\int_{X-\widetilde{X}}^YFq(X,Y')\ {\rm d}Y'\bigg) \nonumber \\
= &\fr{1+(\tan\fr{z}{2}-2\iota c(\theta)\theta)^2}{1+\tan^2\fr{z}{2}}q(X,X-\widetilde{X})\exp\bigg(\int_{X-\widetilde{X}}^YFq(X,Y')\ {\rm d}Y'\bigg) \nonumber \\
\les &q(\widehat{X},X-\widetilde{X})\exp\bigg(\int_{\widehat{X}}^XGp(X',X-\widetilde{X})\ {\rm d}X'\bigg)\cdot \exp\bigg(\int_{X-\widetilde{X}}^YFq(X,Y')\ {\rm d}Y'\bigg)  \nonumber \\
\les & N_0\exp\bigg(2\tilde{F} C_0(\iota, \bar{J}, T, N_0)\bigg),
\end{align}
and
\begin{align}\label{2.57}
q(X,Y)=&q(\widehat{X},Y)\exp\bigg(\int_{\widehat{X}}^XGp(X',Y)\ {\rm d}X'\bigg) \nonumber \\
\les & N_0\exp\bigg(2\tilde{F} C_0(\iota, \bar{J}, T, N_0)\bigg).
\end{align}
The analysis are the same for the point $(X,Y)\in\Omega_{1}^1$ and $(X,Y)\in\Omega_{2}^1$. Hence there exist two positive constants $M_1$ and $N_1$ depending only on the boundary data on $\Gamma_0$ and $\bar{J}(T)$ such that
\begin{align}\label{2.58}
0<M_1\leq \max_{(X,Y)\in\Omega^1}\{p(X,Y), q(X,Y)\}\leq N_1.
\end{align}

Assume that $l\geq1$ is the maximum integer such that
$$
\bigcup_{i=0}^{l-1}\Omega^i\subset\widetilde{\Omega}_T,
$$
and denote
$$
\Omega'=\widetilde{\Omega}_T\setminus\bigg(\bigcup_{i=0}^{l-1}\Omega^i\bigg).
$$
We repeat the above process to obtain
\begin{align}\label{2.59}
0<M_i\leq \max_{(X,Y)\in\Omega^i}\{p(X,Y), q(X,Y)\}\leq N_i, \ \ i=1,2,\cdots,l-1,
\end{align}
where $M_i, N_i$ are positive constants depending only on the boundary data on $\Gamma_0$ and $\bar{J}(T)$.
Set
\begin{align}\label{2.60}
(X_l,Y_l)=
\left\{
\begin{array}{l}
(\widehat{X}+k\widetilde{X},k\widetilde{X}),\qquad \quad  \ l=2k+1,\\
(k\widetilde{X},\widehat{X}+(k-1)\widetilde{X}),\ \ l=2k,
\end{array}
\right.
\end{align}
Then it follows by \eqref{2.59} that
\begin{align}\label{2.61}
0<M_{l-1}\leq \max_{(X,Y)\in\Gamma_l}\{p(X,Y), q(X,Y)\}\leq N_{l-1},
\end{align}
where $\Gamma_l=\Gamma_{l1}\cup\Gamma_{l2}$,
$$
\Gamma_{l1}=\{(X,Y):\ X=X_l,\ Y_l\leq Y\leq X_l\},\quad \Gamma_{l2}=\{(X,Y):\ X_l\leq X\leq Y_l+\widetilde{X},\ Y=Y_l\},
$$
which are two boundaries of the region $\Omega^l$. Using the same argument as before, one can show that
there exist two positive constants $M_l$ and $N_l$ depending only on the boundary data on $\Gamma_0$ and $\bar{J}(T)$ such that
\begin{align}\label{2.62}
0<M_l\leq \max_{(X,Y)\in\Omega'}\{p(X,Y), q(X,Y)\}\leq N_l.
\end{align}
By setting
$$
M=\min\{M_0,M_1,\cdots, M_l\},\quad N=\max\{N_0,N_1,\cdots, N_l\},
$$
we complete the proof of the lemma.
\end{proof}

\subsection{Global existence for the wave equation}\label{S2.6}

This subsection is devoted to verifying that $\theta(x,t)$ constructed above is a weak solution for the wave equation \eqref{2.1} on the region $\Omega_T$.

We first claim that the function $\theta(x,t)$ is H\"older continuous in both $x$ and $t$ with exponent $1/2$. To prove this claim, for any point $(\xi,\tau)\in\Omega_T$, we consider the forward characteristic $t\mapsto x_+(t;\xi,\tau)$. By construction, this curve is parameterized by the function $X\mapsto (t(X,\overline{Y}), x(X,\overline{Y}))$ for a constant $\overline{Y}$ depending on $(\xi,\tau)$.
Integrating along this forward characteristic from $(\xi_0,\tau_0)$ to $(\xi,\tau)$ and using \eqref{2.13}, \eqref{2.14}, \eqref{2.17} and \eqref{2.18} gives
\begin{align}\label{2.63}
&\int_{\tau_0}^\tau[\theta_t+c(\theta)\theta_x]^2(x_+(t;\xi,\tau),t)\ {\rm d}t=\int_{X_0}^{X_\tau} (2cX_xu_X)^2(2X_t)^{-1}\ {\rm d}X \nonumber \\
=& \int_{X_0}^{X_\tau} (2cX_xu_X)^2(2cX_x)^{-1}\ {\rm d}X =\int_{X_0}^{X_\tau} 2c\big(p\cos^2\fr{w}{2}\big)^{-1}\cdot\bigg(\fr{p}{2c}\sin\fr{w}{2}\cos\fr{w}{2}\bigg)^2 \ {\rm d}X
\nonumber \\
=& \int_{X_0}^{X_\tau} \fr{p}{2c}\sin^2\fr{w}{2} \ {\rm d}X\leq \fr{1}{2C_L}\int_{X_0}^{X_\tau} p (X,\overline{Y}) \ {\rm d}X\leq C,
\end{align}
where $C$ is positive constant depending only on $\tau$, $(\xi_0,\tau_0)$ is a point on $t=0$ or $x=0$ satisfying $\xi_0=x_+(\tau_0;\xi,\tau)$, and $(x(X_\tau,\overline{Y}), t(X_\tau,\overline{Y}))=(\xi, \tau)$, $(x(X_0,\overline{Y}), t(X_0,\overline{Y}))=(\xi_0, \tau_0)$.
Similarly, one integrates along the backward characteristic $t\mapsto x_-(t;\xi,\tau)$ from $(\xi_\pi ,\tau_\pi )$ to $(\xi,\tau)$
and noting $X=Const.$ on this kind of characteristics to acquire
\begin{align}\label{2.64}
&\int_{\tau_\pi }^\tau[\theta_t-c(\theta)\theta_x]^2(x_-(t;\xi,\tau),t)\ {\rm d}t\leq C.
\end{align}
Combining \eqref{2.63} and \eqref{2.64} and employing the boundary conditions of $\theta$ on $x=0,\pi$,
we can obtain that $\theta(x,t)$ is H\"older continuous with exponent $1/2$.
Moreover, it is concluded that all characteristic curves are $C^1$ with H\"older continuous derivative.

Next we check that
\begin{align}\label{2.65}
\int_{0}^T\int_{0}^\pi\bigg(\theta_t\varphi_t-(c(\theta)\varphi)_xc(\theta)\theta_x-\theta_t\varphi -J\varphi\bigg)\ {\rm d}x{\rm d}t +\int_{0}^T(c^2\varphi\theta_x)\bigg|_{x=0}^{x=\pi}\ {\rm d}t=0,
\end{align}
for any test function $\varphi\in\mathcal{F}$, where
\begin{align}\label{2.66}
\mathcal{F}:=\bigg\{\varphi\in C^\infty((0,\pi)\times(0,T)):\ \pa_{t}^i\pa_{x}^j\varphi\bigg|_{t=0,T}=0,\ \ \forall\ i,j=0,1,2\cdots \bigg\}.
\end{align}
To obtain \eqref{2.65}, we need to show that
\begin{align}\label{2.67}
0=&\int_{0}^T\int_0^\pi \bigg(\varphi_t[(\theta_t+c\theta_x)+(\theta_t-c\theta_x)] -(c(\theta)\varphi)_x[(\theta_t+c\theta_x)-(\theta_t-c\theta_x)] \nonumber \\
&-2\theta_t\varphi -2J\varphi\bigg)\ {\rm d}x{\rm d}t +\int_{0}^T2(c^2\varphi\theta_x)\bigg|_{x=0}^{x=\pi}\ {\rm d}t \nonumber \\
=& \int_{0}^T\int_0^\pi \bigg([\varphi_t-(c\varphi)_x](\theta_t+c\theta_x) +[\varphi_t+(c\varphi)_x](\theta_t-c\theta_x) \nonumber \\
&-2\theta_t\varphi -2J\varphi \bigg)\ {\rm d}x{\rm d}t +\int_{0}^T2(c^2\varphi\theta_x)\bigg|_{x=0}^{x=\pi}\ {\rm d}t \nonumber \\
=& \int_{0}^T\int_0^\pi \bigg([\varphi_t-(c\varphi)_x]R +[\varphi_t+(c\varphi)_x]S  -2\theta_t\varphi -2J\varphi \bigg)\ {\rm d}x{\rm d}t +\int_{0}^T2(c^2\varphi\theta_x)\bigg|_{x=0}^{x=\pi}\ {\rm d}t \nonumber \\
=& \int_{0}^T\int_0^\pi \bigg(-2cY_x\varphi_YR +2cX_x\varphi_XS +c'(\theta_XX_x+\theta_YY_x)\varphi(S-R) \nonumber \\
& -2(\theta_XX_t+\theta_YY_t)\varphi -2J\varphi \bigg)\ {\rm d}x{\rm d}t +\int_{0}^T2[c^2\varphi(\theta_XX_x+\theta_YY_x)]\ {\rm d}t\bigg|_{x=0}^{x=\pi}.
\end{align}
Noting that the identities
\begin{align*}
\fr{1}{1+R^2}=\cos^2\fr{w}{2},\ \ \fr{1}{1+S^2}=\cos^2\fr{z}{2},\ \ \fr{R}{1+R^2}=\fr{\sin w}{2},\ \ \fr{S}{1+S^2}=\fr{\sin z}{2},
\end{align*}
and
\begin{align*}
{\rm d}x{\rm d}t=\fr{pq}{2c(1+R^2)(1+S^2)}{\rm d}X{\rm d}Y=\fr{pq}{2c}\cos^2\fr{w}{2}\cos^2\fr{z}{2}{\rm d}X{\rm d}Y,
\end{align*}
we use \eqref{2.13}, \eqref{2.16}, \eqref{2.17} and \eqref{2.18} to rewrite the integral in \eqref{2.67} as
\begin{align}\label{2.68}
&\iint_{\widetilde{\Omega}_T}\bigg(\fr{p\sin w}{2}\varphi_Y +\fr{q\sin z}{2}\varphi_X\bigg)\ {\rm d}X{\rm d}Y +\iint_{\widetilde{\Omega}_T}\bigg\{\fr{c'pq}{8c^2}[\cos(w+z)-1]\nonumber \\
&\quad -\fr{pq}{4c}\bigg(\sin w\cos^2\fr{z}{2} +\sin z\cos^2\fr{w}{2}\bigg) -\fr{pq}{c}J\cos^2\fr{w}{2}\cos^2\fr{z}{2}\bigg\}\varphi\ {\rm d}X{\rm d}Y \nonumber \\
&+\int_{(\widehat{X},\phi(\widehat{X}))}^{(X_\pi ,X_\pi -\widetilde{X})}2c^2\varphi\bigg(\fr{\sin w}{4c}p\cdot\fr{1+R^2}{p}-\fr{\sin z}{4c}q\cdot\fr{1+S^2}{q}\bigg)\bigg\{\fr{1+\cos w}{4c}p\ {\rm d}X+\fr{1+\cos z}{4c}q\ {\rm d}Y\bigg\} \nonumber \\
&-\int_{(0,0)}^{(X_0,X_0)}2c^2\varphi\bigg(\fr{\sin w}{4c}p\cdot\fr{1+R^2}{p}-\fr{\sin z}{4c}q\cdot\fr{1+S^2}{q}\bigg)\bigg\{\fr{1+\cos w}{4c}p\ {\rm d}X+\fr{1+\cos z}{4c}q\ {\rm d}Y\bigg\},
\end{align}
which together with the boundary conditions in \eqref{2.23} yields
\begin{align}\label{2.69}
&\iint_{\widetilde{\Omega}_T}\bigg(\fr{p\sin w}{2}\varphi_Y +\fr{q\sin z}{2}\varphi_X\bigg)\ {\rm d}X{\rm d}Y +\iint_{\widetilde{\Omega}_T}\bigg\{\fr{c'pq}{8c^2}[\cos(w+z)-1] \nonumber \\
&\quad -\fr{pq}{4c}\bigg(\sin w\cos^2\fr{z}{2} +\sin z\cos^2\fr{w}{2}\bigg) -\fr{pq}{c}J\cos^2\fr{w}{2}\cos^2\fr{z}{2}\bigg\}\varphi\ {\rm d}X{\rm d}Y \nonumber \\
&+\int_{\widehat{X}}^{X_\pi } p\sin w \varphi(X, X-\widetilde{X})\ {\rm d}X  -\int_{0}^{X_0}p\sin w \varphi(X, X)\ {\rm d}X,
\end{align}
where $(X_0, X_0)$ and $(X_\pi , X_\pi -\widetilde{X})$ are the points in the $(X,Y)$ plane transformed from $(0,T)$ and $(\pi,T)$, respectively.
By means of Green's theorem and boundary conditions, one gets
\begin{align}\label{2.70}
&\iint_{\widetilde{\Omega}_T}\bigg(\fr{p\sin w}{2}\varphi_Y +\fr{q\sin z}{2}\varphi_X\bigg)\ {\rm d}X{\rm d}Y \nonumber \\
= &\iint_{\widetilde{\Omega}_T}\bigg\{\bigg(\fr{p\sin w}{2}\varphi\bigg)_Y +\bigg(\fr{q\sin z}{2}\varphi\bigg)_X\bigg\}\ {\rm d}X{\rm d}Y -\iint_{\widetilde{\Omega}_T}\fr{1}{2}[(p\sin w)_Y+(q\sin z)_X]\varphi\ {\rm d}X{\rm d}Y \nonumber \\
=& \int_{\pa \widetilde{\Omega}_T}-\fr{p\sin w}{2}\varphi\ {\rm d}X +\fr{q\sin z}{2}\varphi\ {\rm d}Y -\iint_{\widetilde{\Omega}_T}\fr{1}{2}[(p\sin w)_Y+(q\sin z)_X]\varphi\ {\rm d}X{\rm d}Y \nonumber \\
=& -\int_{\widehat{X}}^{X_\pi } p\sin w \varphi(X, X-\widetilde{X})\ {\rm d}X  +\int_{0}^{X_0}p\sin w \varphi(X, X)\ {\rm d}X \nonumber \\
&- \iint_{\widetilde{\Omega}_T}\fr{1}{2}[(p\sin w)_Y+(q\sin z)_X]\varphi\ {\rm d}X{\rm d}Y.
\end{align}
According to \eqref{2.17}, we directly compute to achieve
\begin{align}\label{2.71}
&\fr{1}{2}[(p\sin w)_Y+(q\sin z)_X] \nonumber \\
= &\fr{c'pq}{8c^2}[\cos(w+z)-1]  -\fr{pq}{4c}\bigg(\sin w\cos^2\fr{z}{2} +\sin z\cos^2\fr{w}{2}\bigg) -\fr{pq}{c}J\cos^2\fr{w}{2}\cos^2\fr{z}{2}.
\end{align}
Inserting \eqref{2.70} and \eqref{2.71} into \eqref{2.69} leads to
\begin{align}\label{2.72}
&\iint_{\widetilde{\Omega}_T}\bigg(\fr{p\sin w}{2}\varphi_Y +\fr{q\sin z}{2}\varphi_X\bigg)\ {\rm d}X{\rm d}Y +\iint_{\widetilde{\Omega}_T}\bigg\{\fr{c'pq}{8c^2}[\cos(w+z)-1] \nonumber \\
&\quad -\fr{pq}{4c}\bigg(\sin w\cos^2\fr{z}{2} +\sin z\cos^2\fr{w}{2}\bigg) -\fr{pq}{c}J\cos^2\fr{w}{2}\cos^2\fr{z}{2}\bigg\}\varphi\ {\rm d}X{\rm d}Y \nonumber \\
&+\int_{\widehat{X}}^{X_\pi } p\sin w \varphi(X, X-\widetilde{X})\ {\rm d}X  -\int_{0}^{X_0}p\sin w \varphi(X, X)\ {\rm d}X=0,
\end{align}
which completes the proof of \eqref{2.65}.

Moreover, we can repeat the proof process of \eqref{2.73} to show that for any $t\in[0,T]$.
\begin{align}\label{2.73b}
E(t)=\int_0^\pi (\theta_{t}^2+c^2(\theta)\theta_{x}^2)(x,t)\ {\rm d}x +2B(\theta(\pi,t))\leq C_E.
\end{align}
In sum, there has
\begin{lem}
Let $T>0$ be any fixed number and $J(x,t)$ be any $C^\alpha$ function over $\Omega_T$. Assume that the assumptions on initial and boundary conditions in Theorem \ref{thm} hold. Then there exists a weak solution of \eqref{2.1} over $\Omega_T$ with bounded energy $E(t)\leq C_E$ for some $C_E$ depending on $E(0)$ and $J$.
\end{lem}

Finally, for this subsection, we give some comments on the map from $J(x,t)$ to $J(X,Y)$. As pointed out in \cite{CHL20}, one can use a similar method in Bressan, Chen and Zhang \cite{BCZ2} for variational wave equation to show that the uniqueness of forward and backward characteristics for \eqref{2.1}, that is the uniqueness of the $(X,Y)$ coordinates. This uniqueness leads to the uniqueness of $(x_m(X,Y),t_m(X,Y))=(x_n(X,Y),t_n(X,Y))$. Hence, for any given $J(x,t)\in C^\alpha(\Omega_T)$, the solution constructed previously satisfies system \eqref{2.17} with a unique source term $\tilde{J}(X,Y)=J(x(X,Y),t(X,Y))$. Moreover, one can check that $\tilde{J}(X,Y)$ is $L^\infty$. The details can be found in  \cite{CHL20}.

\section{Existence of the coupled system}\label{S3}
In this section, we use the Schauder fixed point theorem to show the global existence of weak solutions to the initial-boundary value problem of the coupled system \eqref{1.1}.

\subsection{Preliminaries for the heat equation}\label{S3.1}

For smooth solutions of \eqref{1.1}, the variable $J=u_x+\theta_t$ satisfies
$$
J_t=J_{xx}+\theta_{tt},
$$
which together with the wave equation in \eqref{1.1} gets
\begin{align}\label{3.1}
J_t-J_{xx}=c(\theta)(c(\theta)\theta_x)_x-\theta_t -J.
\end{align}
According to the initial and boundary conditions \eqref{1.3}-\eqref{1.5}, one obtains
\begin{align}\label{3.2}
J(x,0)=J_0(x):=u_{0}'(x)+\theta_1(x)\in C^\alpha([0,\pi]),\quad J_x(0,t)= J_x(\pi, t)=0.
\end{align}

Let $G_0(x,t;\xi,\tau)$ be the fundamental solution of the heat equation
\begin{align}\label{3.5}
G_0(x,t;\xi,\tau)=\fr{1}{2\sqrt{\pi(t-\tau)}}\exp\bigg(-\fr{(x-\xi)^2}{4(t-\tau)}\bigg),\ \ (t\geq\tau).
\end{align}
Then
\begin{align}\label{3.6}
G(x,t;\xi,\tau)=\fr{1}{\pi}\sum_{n=-\infty}^{\infty}\bigg[&G_0\bigg(\fr{x}{\pi},\fr{t}{\pi^2}; 2n+\fr{\xi}{\pi}, \fr{\tau}{\pi^2}\bigg) -G_0\bigg(\fr{x}{\pi},\fr{t}{\pi^2};2n-\fr{\xi}{\pi}, \fr{\tau}{\pi^2}\bigg)\bigg],
\end{align}
and
\begin{align}\label{3.7}
N(x,t;\xi,\tau)=\fr{1}{\pi}\sum_{n=-\infty}^{\infty}\bigg[&G_0\bigg(\fr{x}{\pi},\fr{t}{\pi^2}; 2n+\fr{\xi}{\pi}, \fr{\tau}{\pi^2}\bigg) +G_0\bigg(\fr{x}{\pi},\fr{t}{\pi^2};2n-\fr{\xi}{\pi}, \fr{\tau}{\pi^2}\bigg)\bigg],
\end{align}
are, respectively, the Green function and the Neumann function for the first and the second initial-boundary value problem of the heat equation.  See Li, Yu and Shen \cite{Li-Yu-Shen1, Li-Yu-Shen2} for details. As a function of $(x,t)$, $G(x,t;\xi,\tau)$ satisfies the equation $G_t=G_{xx}$ for $t>\tau$ and $G=0$ for $x=0,\pi$, while as a function of $(\xi,\tau)$, $G(x,t;\xi,\tau)$ satisfies the adjoint equation $G_\tau=-G_{\xi\xi}$ for $\tau<t$ and $G=0$ for $\xi=0,\pi$. Moreover, as a function of $(x,t)$, $N(x,t;\xi,\tau)$ satisfies the equation $N_t=N_{xx}$ for $t>\tau$ and $N_x=0$ for $x=0,\pi$, while as a function of $(\xi,\tau)$, $N(x,t;\xi,\tau)$ satisfies the adjoint equation $N_\tau=-N_{\xi\xi}$ for $\tau<t$ and $N_\xi=0$ for $\xi=0,\pi$.
Furthermore, we have
\begin{align}\label{3.8}
\fr{\pa G}{\pa x}=-\fr{\pa N}{\pa \xi},\quad \fr{\pa G}{\pa \xi}=-\fr{\pa N}{\pa x}.
\end{align}

Set
\begin{align}\label{3.9}
W_\sigma(z,t)=t^{-\fr{\sigma}{2}}\exp\bigg(-\fr{z^2}{16t}\bigg).
\end{align}
A direct calculation arrives at
\begin{align}\label{3.10}
\int_{\mathbb{R}}W_\sigma(\xi,t)\ {\rm d}\xi=4\sqrt{\pi}t^{\fr{1-\sigma}{2}}, \ \ (t>0),
\end{align}
and
\begin{align}\label{3.11}
\int_{0}^t\int_{\mathbb{R}}W_\sigma(\xi, t-\tau)\ {\rm d}\xi{\rm d}\tau=\fr{8\sqrt{\pi}}{3-\sigma}t^{\fr{3-\sigma}{2}}, \ \ (\sigma<3).
\end{align}
In addition, one also has
\begin{align}\label{3.12}
|\eta|^\beta\exp\bigg(-\fr{\eta}{4}\bigg)\les\exp\bigg(-\fr{\eta}{16}\bigg),
\end{align}
for $\eta\geq0$, where $\beta$ is an arbitrary nonnegative real number.

By using the Neumann function $N(x,t;\xi,\tau)$, the solution of the first initial-boundary value problem \eqref{3.1}, \eqref{3.2} can be expressed as
\begin{align}\label{d}
J(x,t)
= & \int_{0}^\pi N(x,t;\xi,0)J_0(\xi)\ {\rm d}\xi -\int_{0}^t\int_{0}^\pi N(x,t;\xi,\tau)(\theta_\tau+J)(\xi,\tau)  \ {\rm d}\xi{\rm d}\tau \nonumber \\
&  +\int_{0}^t N(x,t;\xi,\tau)c^2(\theta)\theta_\xi(\xi,\tau)\bigg|_{\xi=0}^{\xi=\pi}\ {\rm d}\tau -\int_{0}^t\int_{0}^\pi N(x,t;\xi,\tau)cc'\theta_{\xi}^2(\xi,\tau)\ {\rm d}\xi{\rm d}\tau \nonumber \\
& -\int_{0}^t\int_{0}^\pi\pa_\xi N(x,t;\xi,\tau)c^2\theta_{\xi}(\xi,\tau)\ {\rm d}\xi{\rm d}\tau.
\end{align}

\subsection{The existence of fixed point}\label{S3.2}

For any given function $J(x,t)\in C^\alpha(\Omega_T)$, we know by previous analysis that there exists a weak solution $\theta(x,t)=\theta^J(x,t)$ of the wave equation \eqref{2.1}. In view of \eqref{d}, we can define a function $\mathcal{M}(J)$ on $\Omega_T$
\begin{align}\label{3.14}
 \mathcal{M}(J)(x,t)
= & \int_{0}^\pi N(x,t;\xi,0)J_0(\xi)\ {\rm d}\xi -\int_{0}^t\int_{0}^\pi N(x,t;\xi,\tau)(\theta_\tau+J)(\xi,\tau)  \ {\rm d}\xi{\rm d}\tau \nonumber \\
&  +\int_{0}^t N(x,t;\xi,\tau)c^2(\theta)\theta_\xi(\xi,\tau)\bigg|_{\xi=0}^{\xi=\pi}\ {\rm d}\tau -\int_{0}^t\int_{0}^\pi N(x,t;\xi,\tau)cc'\theta_{\xi}^2(\xi,\tau)\ {\rm d}\xi{\rm d}\tau \nonumber \\
&  -\int_{0}^t\int_{0}^\pi\pa_\xi N(x,t;\xi,\tau)c^2\theta_{\xi}(\xi,\tau)\ {\rm d}\xi{\rm d}\tau.
\end{align}
It follows by \eqref{d} that $\mathcal{M}(J)(x,t)$ is a weak solution of
\begin{align}\label{3.14a}
\mathcal{M}_t-\mathcal{M}_{xx}=c(\theta)(c(\theta)\theta_x)_x-\theta_t-J.
\end{align}
Furthermore, we have a map
\begin{align}\label{3.15}
\mathcal{T}:\ J(x,t)\rightarrow \mathcal{M}(J)(x,t),
\end{align}
on $C^\alpha(\Omega_T)$. Following Chen-Huang-Liu \cite{CHL20}, we define a set $\mathcal{K}$ for some constants $\delta$ and $K$
\begin{align}\label{3.16}
\mathcal{K}=\bigg\{&J(x,t)\big|\ \|J(x,t)-J^0(x,t)\|_{C^\alpha(\Omega_\delta)}\leq K,  J(x,0)=J_0(x),\ J_x(0,t)=J_x(\pi,t)=0, \text{ a.e. } \bigg\},
\end{align}
where $\Omega_\delta=[0,\pi]\times[0,\delta]$ and
$$
J^0(x,t)=\int_{0}^{\pi}N(x,t;\xi,0)J_0(\xi)\ {\rm d}\xi.
$$
We first show that, for a given large $K$, the map $\mathcal{T}$ has a fixed point on $\mathcal{K}$ if $\delta$ is sufficiently small. Next we derive the energy estimate to select a constant $K$ that only depends on the initial value. Based on this constant $K$, we can fix the small number $\delta>0$ and then extend the existence on $\Omega_\delta$ to $\Omega_T$ in finite steps.

We apply the Schauder fixed point theorem to show the existence of fixed points on $\mathcal{K}$ for sufficiently small $\delta$. It is clear to see that $\mathcal{K}$ is a compact set in $L^\infty$. To use the Schauder fixed point theorem, it suffices to check that $\mathcal{T}$ maps from $\mathcal{K}$ to itself and is continuous under the $L^\infty$ norm.

\begin{lem}\label{lem}
$\mathcal{T}$ maps from $\mathcal{K}$ to itself.
\end{lem}
\begin{proof}
According to the properties of the Neumann function $N(x,t;\xi,\tau)$, we know that $\mathcal{M}(J)(x,0)=J_0(x)$ and $\partial_x\mathcal{M}(J)(0,t)=\partial_x\mathcal{M}(J)(\pi,t)=0$. To finish the proof of the lemma, it suffices to verify that $L_{0,1,2,3}(x,t)$ are $C^\alpha$ functions, where
\begin{align*}
L_0(x,t)&=\int_{0}^t N(x,t;\xi,\tau)c^2(\theta)\theta_\xi(\xi,\tau)\bigg|_{\xi=0}^{\xi=\pi}\ {\rm d}\tau, \\
L_1(x,t)&=\int_{0}^t\int_{0}^{\pi}N(x,t;\xi,\tau)(\theta_\tau+J)(\xi,\tau)  \ {\rm d}\xi{\rm d}\tau, \\
L_2(x,t)&=\int_{0}^t\int_{0}^{\pi}N(x,t;\xi,\tau)cc'\theta_{\xi}^2(\xi,\tau)\ {\rm d}\xi{\rm d}\tau, \\
L_3(x,t)&=\int_{0}^t\int_{0}^{\pi}\pa_\xi N(x,t;\xi,\tau)c^2\theta_{\xi}(\xi,\tau)\ {\rm d}\xi{\rm d}\tau.
\end{align*}
We just only consider the function $L_3$, and the functions $L_{0, 1, 2}$ can be discussed analogously.
In view of the expression of $L_3$, it is sufficient to show that $L_3$ is $C^\alpha$ continuous with respect to $x$, that is
\begin{align}\label{3.17}
\int_{0}^t\int_{0}^{\pi}\fr{|\pa_\xi N(x_2,t;\xi,\tau) -\pa_\xi N(x_1,t;\xi,\tau)|}{(x_2-x_1)^\alpha}c^2|\theta_{\xi}|(\xi,\tau)\ {\rm d}\xi{\rm d}\tau\leq Ct^\nu,
\end{align}
for any $x_1<x_2\in[0,\pi]$, where $C$ is positive constant and $\nu\in(0,1)$. Recall the expression of $N(x,t;\xi,\tau)$ in \eqref{3.6}, one has
\begin{align}\label{3.18}
\pa_\xi N(x,t;\xi,\tau)=&\sum_{n=-\infty}^{\infty}\fr{1}{2\sqrt{\pi(t-\tau)}}\bigg[\exp\bigg( -\fr{(x-2n\pi-\xi)^2}{4(t-\tau)}\bigg)\cdot \bigg(-\fr{x-2n\pi-\xi}{2(t-\tau)}\bigg) \nonumber \\
&+\exp\bigg( -\fr{(x-2n
\pi+\xi)^2}{4(t-\tau)}\bigg)\cdot \bigg(-\fr{x-2n\pi+\xi)}{2(t-\tau)}\bigg)\bigg],
\end{align}
from which we obtain
\begin{align}\label{3.19}
\pa_\xi N(x_2,t;\xi,\tau)-\pa_\xi N(x_1,t;\xi,\tau)=\sum_{n=-\infty}^{\infty}(L_{31}^{(n)}+L_{32}^{(n)}),
\end{align}
where
\begin{align*}
L_{31}^{(n)}=&\fr{1}{4\sqrt{\pi}(t-\tau)^{\fr{3}{2}}}\bigg[\exp\bigg( -\fr{(x_1-2n\pi-\xi)^2}{4(t-\tau)}\bigg)\cdot (x_1-2n\pi-\xi ) \nonumber \\
&-\exp\bigg( -\fr{(x_2-2n\pi-\xi)^2}{4(t-\tau)}\bigg)\cdot (x_2-2n\pi-\xi ) \bigg],
\end{align*}
and
\begin{align*}
L_{32}^{(n)}=&\fr{1}{4\sqrt{\pi}(t-\tau)^{\fr{3}{2}}}\bigg[\exp\bigg( -\fr{(x_1-2n\pi+\xi)^2}{4(t-\tau)}\bigg)\cdot (x_1-2n\pi+\xi ) \nonumber \\
&-\exp\bigg( -\fr{(x_2-2n\pi+\xi)^2}{4(t-\tau)}\bigg)\cdot (x_2-2n\pi+\xi ) \bigg],
\end{align*}
We first fix $n\in \mathbb{Z}$ and consider the integral
\begin{align}\label{3.20}
I^{(n)}=&\int_{0}^t\int_{0}^{\pi} \fr{|L_{31}^{(n)}|}{(x_2-x_1)^\alpha}c^2|\theta_{\xi}|(\xi,\tau)\ {\rm d}\xi{\rm d}\tau \nonumber \\
=& \int_{0}^t\int_{B_1} \fr{|L_{31}^{(n)}|}{(x_2-x_1)^\alpha}c^2|\theta_{\xi}|(\xi,\tau)\ {\rm d}\xi{\rm d}\tau + \int_{0}^t\int_{B_2} \fr{|L_{31}^{(n)}|}{(x_2-x_1)^\alpha}c^2|\theta_{\xi}|(\xi,\tau)\ {\rm d}\xi{\rm d}\tau \nonumber \\
=:& I_1^{(n)}+I_2^{(n)},
\end{align}
where
\begin{align*}
B_1:\ \big|x_1-2n\pi-\xi\big|<\big|x_2-2n\pi-\xi\big|, \\
B_2:\ \big|x_1-2n\pi-\xi\big|>\big|x_2-2n\pi-\xi\big|.
\end{align*}
By the definitions of $B_1$ and $B_2$, we find that
\begin{align}\label{3.21}
\big|x_2-x_1\big|\leq  \big|x_2-2n\pi-\xi\bigg| +\big|x_1-2n\pi-\xi\big|
\leq 2\bigg|x_2-2n\pi-\xi\big|,\quad {\rm on}\ B_1,
\end{align}
and
\begin{align}\label{3.22}
\big|x_2-x_1\big|\leq\big|x_2-2n\pi-\xi\big| +\big|x_1-2n\pi-\xi\big|
\leq 2\big|x_1-2n\pi-\xi\big|,\quad {\rm on}\ B_2,
\end{align}
Then we handle $I_1^{(n)}$ as follows
\begin{align}\label{3.23}
I_1^{(n)}\leq &\int_{0}^t\int_{B_1}\fr{ c^2}{4\sqrt\pi}\cdot\fr{1}{(x_2-x_1)^\alpha(t-\tau)^{\fr{3}{2}}}\nonumber \\
&\quad \cdot (x_2-x_1)\exp\bigg(-\fr{(x_2-2n\pi-\xi)^2}{4(t-\tau)}\bigg) |\theta_{\xi}|(\xi,\tau)\ {\rm d}\xi{\rm d}\tau \nonumber \\
&+\int_{0}^t\int_{B_1}\fr{c^2}{4\sqrt{\pi}}\cdot\fr{1}{(x_2-x_1)^\alpha(t-\tau)^{\fr{3}{2}}} \nonumber \\
& \cdot \big|x_1-2n\pi-\xi\big| \cdot\exp\bigg(-\fr{(x_1-2n\pi-\xi)^2}{4(t-\tau)}\bigg) \nonumber \\
&\cdot\bigg| 1-\exp\bigg(-\fr{[(x_2-2n\pi-\xi)^2 -(x_1-2n\pi-\xi)^2]}{4(t-\tau)}\bigg) \bigg| |\theta_{\xi}|(\xi,\tau)\ {\rm d}\xi{\rm d}\tau \nonumber \\
=:& I_{11}^{(n)}+I_{12}^{(n)}.
\end{align}
The term $I_2^{(n)}$ can be treated similarly by symmetry. For $I_{11}^{(n)}$, one gets by \eqref{3.21}
\begin{align}\label{3.24}
I_{11}^{(n)}\les & \int_{0}^t\int_{B_1} \fr{1}{(t-\tau)^{\fr{3}{2}}}\bigg|\fr{x_2}{\pi}-2n-\fr{\xi}{\pi}\bigg|^{1-\alpha} \nonumber \\
&\cdot \exp\bigg(-\fr{(x_2-2n\pi-\xi)^2}{4(t-\tau)}\bigg) |\theta_{\xi}|(\xi,\tau)\ {\rm d}\xi{\rm d}\tau  \nonumber \\
\leq & \bigg\{\int_{0}^t\int_{B_1} \fr{1}{(t-\tau)^{3-2r_1}}\bigg|\fr{x_2}{\pi}-2n-\fr{\xi}{\pi}\bigg|^{2-2\alpha} \nonumber \\ &\cdot\exp\bigg(-\fr{(x_2-2n\pi-\xi)^2}{2(t-\tau)}\bigg) \ {\rm d}\xi{\rm d}\tau \bigg\}^{\fr{1}{2}}\cdot \bigg\{\int_{0}^t\int_{B_1} \fr{|\theta_\xi|^2}{(t-\tau)^{2r_1}} \ {\rm d}\xi{\rm d}\tau \bigg\}^{\fr{1}{2}},
\end{align}
where $r_1$ satisfies
\begin{align}\label{3.25}
\fr{1}{4}+\fr{\alpha}{2}<r_1<\fr{1}{2},\quad {\rm and\ then}\ \sigma_1=4-4r_1+2\alpha<3.
\end{align}
Applying the fact
\begin{align*}
&\bigg(\fr{(x_2-2n\pi-\xi)^2}{t-\tau}\bigg)^{1-\alpha} \exp\bigg(-\fr{(x_2-2n\pi-\xi)^2}{2(t-\tau)}\bigg) \\
\les &\exp\bigg(-\fr{(x_2-2n\pi-\xi)^2}{4(t-\tau)}\bigg),
\end{align*}
we have
\begin{align}\label{3.26}
I_{11}^{(n)}\les & \bigg\{\int_{0}^t\int_{B_1} \fr{(t-\tau)^{1-\alpha}}{(t-\tau)^{3-2r_1}} \cdot\exp\bigg(-\fr{(x_2-2n\pi-\xi)^2}{4(t-\tau)}\bigg) \ {\rm d}\xi{\rm d}\tau \bigg\}^{\fr{1}{2}} \nonumber \\
&\cdot \bigg\{\int_{0}^t\fr{1}{(t-\tau)^{2r_1}}\bigg(\int_{0}^{\pi} |\theta_\xi|^2 \ {\rm d}\xi\bigg){\rm d}\tau \bigg\}^{\fr{1}{2}} \nonumber \\
&\les \|\theta_x\|_{L^2([0,\pi])}t^{\fr{1}{2}-r_1}\bigg\{\int_{0}^t\int_{B_1} (t-\tau)^{-\fr{\sigma_1}{2}}\exp\bigg(-\fr{(x_2-2n\pi-\xi)^2}{4(t-\tau)}\bigg) \ {\rm d}\xi{\rm d}\tau \bigg\}^{\fr{1}{2}}.
\end{align}
Note that
\begin{align*}
\bigg(\fr{x}{\pi}-2n-\fr{\xi}{\pi}\bigg)^2\geq \bigg(\fr{x-\xi}{\pi}\bigg)^2 +(2|n|-1)^2-1, \\
\bigg(\fr{x}{\pi}-2n+\fr{\xi}{\pi}\bigg)^2\geq \bigg(\fr{x-\xi}{\pi}\bigg)^2 +(2|n|-2)^2-4,
\end{align*}
we see that
\begin{align}\label{3.27}
(2|n|-1)^2-1\geq0,\ \ (2|n|-2)^2-4\geq0,\ \ {\rm for}\ |n|\geq2.
\end{align}
Then by \eqref{3.11} one acquires
\begin{align}\label{3.28}
&\int_{0}^t\int_{B_1} (t-\tau)^{-\fr{\sigma_1}{2}}\exp\bigg(-\fr{(x_2-2n\pi-\xi)^2}{4(t-\tau)}\bigg) \ {\rm d}\xi{\rm d}\tau \nonumber \\
\leq &\int_{0}^t\int_{\mathbb{R}} (t-\tau)^{-\fr{\sigma_1}{2}}\exp\bigg(-\fr{(x_2-\xi)^2}{4(t-\tau)}\bigg) \ {\rm d}\xi{\rm d}\tau \les t^{\fr{3-\sigma_1}{2}},
\end{align}
for $|n|< 2$, and
\begin{align}\label{3.29}
&\int_{0}^t\int_{B_1} (t-\tau)^{-\fr{\sigma_1}{2}}\exp\bigg(-\fr{(x_2-2n\pi-\xi)^2}{4(t-\tau)}\bigg) \ {\rm d}\xi{\rm d}\tau \nonumber \\
\leq &\int_{0}^t\int_{\mathbb{R}} (t-\tau)^{-\fr{\sigma_1}{2}}\exp\bigg(-\fr{(x_2-\xi)^2}{4(t-\tau)}\bigg) \cdot \exp\bigg(-\fr{(2|n|-1)^2-1}{4(t-\tau)}\bigg) \ {\rm d}\xi{\rm d}\tau \nonumber \\
\leq &\int_{0}^t\int_{\mathbb{R}} (t-\tau)^{-\fr{\sigma_1}{2}}\exp\bigg(-\fr{(x_2-\xi)^2}{4(t-\tau)}\bigg) \cdot \exp\bigg(-\fr{(2|n|-1)^2-1}{4T}\bigg) \ {\rm d}\xi{\rm d}\tau  \nonumber \\
\les & t^{\fr{3-\sigma_1}{2}} \exp\bigg(-\fr{(2|n|-1)^2-1}{4T}\bigg),
\end{align}
for $|n|\geq 2$. Set
\begin{align}\label{3.30}
A_n=\left\{
\begin{array}{l}
1,\qquad \qquad \qquad \qquad \qquad \quad  |n|<2, \\
\dps \exp\bigg(-\fr{(2|n|-1)^2-1}{8T}\bigg),\ |n|\geq2.
\end{array}
\right.
\end{align}
Obviously, the positive series $\dps\sum_{n=-\infty}^\infty A_n$ is convergent.
Putting \eqref{3.28}-\eqref{3.30} into \eqref{3.26} yields
\begin{align}\label{3.31}
I_{11}^{(n)}
&\les \|\theta_x\|_{L^2([0,\pi])}t^{\fr{1}{2}-r_1}\cdot A_nt^{\fr{3-\sigma_1}{4}} \nonumber \\
&= \|\theta_x\|_{L^2([0,\pi])}t^{\fr{1}{2}-r_1}\cdot A_nt^{\fr{3-(4-4r+2\alpha)}{4}} =\|\theta_x\|_{L^2([0,\pi])}t^{\fr{1}{4}-\fr{\alpha}{2}} A_n.
\end{align}

To deal with the term $I_{12}^{(n)}$, we recall the following property:
\begin{align}\label{3.32}
1-e^{-x}\leq \fr{1}{\lambda}x^\lambda,
\end{align}
for all $x\geq0$ and $0<\lambda\leq 1$.
Since
\begin{align*}
&\bigg(x_2-2n\pi-\xi\bigg)^2 -\bigg(x_1-2n\pi-\xi\bigg)^2 \\
=&(x_2-x_1)\cdot2\bigg(\fr{x_1+x_2}{2}-2n\pi-\xi\bigg)\geq0,
\end{align*}
on the region $B_1$, then we use \eqref{3.32} to achieve
\begin{align}\label{3.33}
0\leq & 1-\exp\bigg(-\fr{(x_2-2n\pi-\xi)^2 -(x_1-2n\pi-\xi)^2}{4(t-\tau)}\bigg) \nonumber \\
\leq & \fr{1}{\fr{1}{2}-\mu} \bigg(\fr{(x_2-x_1)(\fr{x_1+x_2}{2}-2n\pi-\xi)} {2(t-\tau)}\bigg)^{\fr{1}{2}-\mu},
\end{align}
for any number $\mu\in(\fr{1}{4},\fr{1-\alpha}{2})$. Thus we have
\begin{align}\label{3.34}
I_{12}^{(n)}\les &\int_{0}^t\int_{B_1}\fr{1}{(x_2-x_1)^\alpha(t-\tau)^{\fr{3}{2}}} \nonumber \\
& \cdot \bigg|x_1-2n\pi-\xi\bigg| \cdot\exp\bigg(-\fr{(x_1-2n\pi-\xi)^2}{4(t-\tau)}\bigg) \nonumber \\
&\cdot \bigg(\fr{(x_2-x_1)(\frac{x_1+x_2}{2}-2n\pi-\xi)} {(t-\tau)}\bigg)^{\fr{1}{2}-\mu} |\theta_{\xi}|(\xi,\tau)\ {\rm d}\xi{\rm d}\tau.
\end{align}
Noting that
\begin{align*}
&(x_2-x_1)\bigg(\frac{x_1+x_2}{2}-2n\pi-\xi\bigg) =(x_2-x_1)\bigg(\fr{x_2-x_1}{2}+x_1-2n\pi-\xi\bigg) \\
=&\fr{1}{2}\big(x_2-x_1\big)^2 +(x_2-x_1)\cdot\bigg(x_1-2n\pi-\xi\bigg),
\end{align*}
one has
\begin{align*}
&\bigg[(x_2-x_1)\bigg(\frac{x_1+x_2}{2}-2n\pi-\xi\bigg) \bigg]^{\fr{1}{2}-\mu} \\
\les & \big(x_2-x_1\big)^{1-2\mu} +\big(x_2-x_1\big)^{\fr{1}{2}-\mu}\cdot \bigg(x_1-2n\pi-\xi\bigg)^{\fr{1}{2}-\mu}.
\end{align*}
Inserting the above into \eqref{3.34} arrives at
\begin{align}\label{3.35}
I_{12}^{(n)}\les &\int_{0}^t\int_{B_1}\fr{(x_2-x_1)^{1-2\mu-\alpha}}{(t-\tau)^{2-\mu}} \cdot \bigg|x_1-2n\pi-\xi\bigg| \nonumber \\
& \ \ \cdot\exp\bigg(-\fr{(x_1-2n\pi-\xi)^2}{4(t-\tau)}\bigg) \cdot |\theta_{\xi}|(\xi,\tau)\ {\rm d}\xi{\rm d}\tau \nonumber \\
& + \int_{0}^t\int_{B_1}\fr{(x_2-x_1)^{\fr{1}{2}-\mu-\alpha}}{(t-\tau)^{2-\mu}} \cdot \bigg|x_1-2n\pi-\xi\bigg|^{\fr{3}{2}-\mu} \nonumber \\
& \ \ \cdot\exp\bigg(-\fr{(x_1-2n\pi-\xi)^2}{4(t-\tau)}\bigg) \cdot |\theta_{\xi}|(\xi,\tau)\ {\rm d}\xi{\rm d}\tau=:I_{12,1}^{(n)}+I_{12,2}^{(n)}.
\end{align}
For $I_{12,1}^{(n)}$, we get by $1-2\mu-\alpha>0$
\begin{align}\label{3.35_2}
I_{12,1}^{(n)}\les &\bigg(\int_{0}^t\int_{B_1}\fr{1}{(t-\tau)^{4-2\mu-2r_2}} \cdot \bigg|x_1-2n\pi-\xi\bigg|^2 \nonumber \\
& \ \ \cdot\exp\bigg(-\fr{(x_1-2n\pi-\xi)^2}{2(t-\tau)}\bigg) \ {\rm d}\xi{\rm d}\tau \bigg)^{\fr{1}{2}}\cdot \bigg(\int_{0}^t\int_{B_1} \fr{|\theta_{\xi}|^2}{(t-\tau)^{2r_2}} \ {\rm d}\xi{\rm d}\tau \bigg)^{\fr{1}{2}} \nonumber \\
\les & \|\theta_x\|_{L^2([0,\pi])}t^{\fr{1}{2}-r_2} \bigg(\int_{0}^t\int_{B_1}t^{-\fr{\sigma_2}{2}} \exp\bigg(-\fr{(x_1-2n\pi-\xi)^2}{4(t-\tau)}\bigg) \ {\rm d}\xi{\rm d}\tau \bigg)^{\fr{1}{2}},
\end{align}
where $r_2$ and $\sigma_2$ satisfy
\begin{align}\label{3.36}
\fr{3}{4}-\mu<r_2<\fr{1}{2},\quad  \sigma_2=6-4r_2-4\mu<3.
\end{align}
Similar to handling \eqref{3.26}, we can obtain
\begin{align}\label{3.37}
I_{12,1}^{(n)}\les  \|\theta_x\|_{L^2([0,\pi])}t^{\fr{1}{2}-r_2} \cdot t^{\fr{3-\sigma_2}{4}}A_n =\|\theta_x\|_{L^2([0,\pi])}t^{\mu-\fr{1}{4}}A_n.
\end{align}
Now for $I_{12,2}^{(n)}$, one acquires by $\fr{1}{2}-\mu-\alpha>0$
\begin{align}\label{3.38}
I_{12,2}^{(n)}\les & \bigg(\int_{0}^t\int_{B_1}\fr{1}{(t-\tau)^{4-2\mu-2r_3}} \cdot \bigg|x_1-2n\pi-\xi\bigg|^{3-2\mu} \nonumber \\
& \ \ \cdot\exp\bigg(-\fr{(x_1-2n\pi-\xi)^2}{2(t-\tau)}\bigg) \ {\rm d}\xi{\rm d}\tau\bigg)^{\fr{1}{2}} \cdot \bigg(\int_{0}^t\int_{B_1} \fr{|\theta_{\xi}|^2}{(t-\tau)^{2r_3}} \ {\rm d}\xi{\rm d}\tau \bigg)^{\fr{1}{2}} \nonumber \\
\les & \|\theta_x\|_{L^2([0,\pi])}t^{\fr{1}{2}-r_3} \bigg(\int_{0}^t\int_{B_1}t^{-\fr{\sigma_3}{2}} \exp\bigg(-\fr{(x_1-2n\pi-\xi)^2}{4(t-\tau)}\bigg) \ {\rm d}\xi{\rm d}\tau \bigg)^{\fr{1}{2}},
\end{align}
where $r_3$ and $\sigma_3$ satisfy
\begin{align}\label{3.39}
\fr{1}{2}-\fr{\mu}{2}<r_3<\fr{1}{2},\quad \sigma_3=5-4r_3-2\mu<3.
\end{align}
Hence we find that
\begin{align}\label{3.40}
I_{12,2}^{(n)}\les \|\theta_x\|_{L^2([0,\pi])}t^{\fr{1}{2}-r_3}\cdot t^{\fr{3-\sigma_3}{4}} A_n =\|\theta_x\|_{L^2([0,\pi])}t^{\fr{\mu}{2}}A_n.
\end{align}
Combining \eqref{3.35}, \eqref{3.37} and \eqref{3.40} leads to
\begin{align}\label{3.41}
I_{12}^{(n)}\les \|\theta_x\|_{L^2([0,\pi])}(t^{\mu-\fr{1}{4}}+t^{\fr{\mu}{2}})A_n.
\end{align}
Thus we sum up \eqref{3.23}, \eqref{3.31} and \eqref{3.41} to conclude that
\begin{align}\label{3.42}
I_{1}^{(n)}\les \|\theta_x\|_{L^2([0,\pi])}(t^{\fr{1}{4}-\fr{\alpha}{2}} +t^{\mu-\fr{1}{4}}+t^{\fr{\mu}{2}})A_n.
\end{align}
The above estimate is also valid for the term $I_{2}^{(n)}$ in \eqref{3.20}. Therefore there holds
\begin{align}\label{3.43}
I^{(n)}=\int_{0}^t\int_{0}^{\pi} \fr{|L_{31}^{(n)}|}{(x_2-x_1)^\alpha}c^2|\theta_{\xi}|(\xi,\tau)\ {\rm d}\xi{\rm d}\tau \les \|\theta_x\|_{L^2([0,\pi])}(t^{\fr{1}{4}-\fr{\alpha}{2}} +t^{\mu-\fr{1}{4}}+t^{\fr{\mu}{2}})A_n,
\end{align}
which together with \eqref{3.30} implies that the series $\dps\sum_{n=-\infty}^\infty I$ is uniform convergent. Similarly, one also has
\begin{align}\label{3.44}
\int_{0}^t\int_{0}^{\pi} \fr{|L_{32}^{(n)}|}{(x_2-x_1)^\alpha}c^2|\theta_{\xi}|(\xi,\tau)\ {\rm d}\xi{\rm d}\tau \les \|\theta_x\|_{L^2([0,\pi])}(t^{\fr{1}{4}-\fr{\alpha}{2}} +t^{\mu-\fr{1}{4}}+t^{\fr{\mu}{2}})A_n.
\end{align}
Then
\begin{align}\label{3.45}
&\int_{0}^t\int_{0}^{\pi}\fr{|\pa_\xi N(x_2,t;\xi,\tau) -\pa_\xi N(x_1,t;\xi,\tau)|}{(x_2-x_1)^\alpha}c^2|\theta_{\xi}|(\xi,\tau)\ {\rm d}\xi{\rm d}\tau \nonumber \\
\leq & \int_{0}^t\int_{0}^{\pi}\fr{1}{(x_2-x_1)^\alpha}\sum_{n=-\infty}^{\infty}(|L_{31}^{(n)}|+|L_{32}^{(n)}|) c^2|\theta_{\xi}|(\xi,\tau)\ {\rm d}\xi{\rm d}\tau \nonumber \\
= & \sum_{n=-\infty}^{\infty}\int_{0}^t\int_{0}^{\pi}\fr{1}{(x_2-x_1)^\alpha}(|L_{31}^{(n)}|+|L_{32}^{(n)}|) c^2|\theta_{\xi}|(\xi,\tau)\ {\rm d}\xi{\rm d}\tau \nonumber \\
\les & \|\theta_x\|_{L^2([0,\pi])}(t^{\fr{1}{4}-\fr{\alpha}{2}} +t^{\mu-\fr{1}{4}}+t^{\fr{\mu}{2}})\sum_{n=-\infty}^{\infty}A_n\leq Ct^\nu,
\end{align}
for some constant $C>0$, where $\nu=\min\{\fr{1}{4}-\fr{\alpha}{2}, \mu-\fr{1}{4}, \fr{\mu}{2}\}>0$. The proof of Lemma \ref{lem} is completed by choosing $3C\delta^\nu<K$.
\end{proof}

\begin{lem}\label{lem-b}
The map $\mathcal{T}$ is continuous under the $L^\infty$ norm.
\end{lem}
\begin{proof}
The proof follows from the fact that $(\theta,w,z,p,q,x,t)(X,Y)$ is Lipschitz continuously dependent on $\tilde{J}(X,Y)$ in the $L^\infty$ distance. The proof is entirely similar to that in Chen-Huang-Liu \cite{CHL20} and we omit it here.
\end{proof}

By means of Lemmas \ref{lem} and \ref{lem-b}, we know by the Schauder
fixed point theorem that there exists a function $J^*(x,t)\in \mathcal{K}$ such that
\begin{align}\label{3.46}
\mathcal{M}(J^*)(x,t)=J^*(x,t).
\end{align}
We fix $J=J^*$ and utilize previous results to obtain that
\begin{align}\label{3.47}
\theta_t(\cdot,t),\ \theta_x(\cdot,t)\in L^2([0,\pi]),\quad J\in C^\alpha([0,\pi]\times[0,\delta]).
\end{align}

Furthermore, we consider the initial-boundary value problem for $u$
\begin{align}\label{3.50}
\left\{
\begin{array}{l}
u_t-u_{xx}=\theta_{tx},\\
u(x,0)=u_0(x),\\
u(0,t)=u(\pi,t)=0.
\end{array}
\right.
\end{align}
According to the properties of the Green function $G(x,t;\xi,\tau)$ given in \eqref{3.7}, the weak solution $u(x,t), (t\in[0,\delta])$ of \eqref{3.50} can be expressed as
\begin{align}\label{3.51}
u(x,t)=\int_{0}^\pi G(x,t;\xi,0)u_0(\xi)\ {\rm d}\xi -\int_{0}^t\int_{0}^\pi \pa_\xi G(x,t;\xi,\tau)\theta_\tau(\xi,\tau)\ {\rm d}\xi{\rm d}\tau.
\end{align}
Clearly, this weak solution $u(x,t)$ satisfies
\begin{align}\label{3.52}
\int_{0}^\delta\int_{0}^\pi\bigg(u\psi_t-(u_x+\theta_t)\psi_x\bigg)\ {\rm d}x{\rm d}t +\int_{0}^\delta (u_x+\theta_t)\psi\bigg|_{x=0}^{x=\pi}\ {\rm d}t=0,
\end{align}
for any test functions $\psi\in\mathcal{F}$. Letting $\psi=\varphi_x$ in \eqref{3.52} gives
\begin{align}\label{3.53}
\int_{0}^\delta\int_{0}^\pi \bigg(u_x\varphi_t+(u_x+\theta_t)\varphi_{xx}\bigg)\ {\rm d}x{\rm d}t -\int_{0}^\delta (u_x+\theta_t)\varphi_x\bigg|_{x=0}^{x=\pi}\ {\rm d}t=0,
\end{align}

Finally, we prove $J=u_x+\theta_t$ in some sense. By \eqref{3.14a}, the function $J$ satisfies
\begin{align}\label{3.48}
\int_{0}^\delta\int_{0}^\pi\bigg(J(\varphi^0_t+\varphi^0_{xx})-[(c\varphi^0)_xc\theta_x+(\theta_t+J)\varphi^0]\bigg)\ {\rm d}x{\rm d}t- \int_{0}^\delta [J\varphi^0_x-c^2\varphi^0\theta_x]\bigg|_{x=0}^{x=\pi}\ {\rm d}t=0,
\end{align}
for any test function $\varphi^0(x,t)\in C^\infty_c([0,\pi]\times[0,\delta])$,
which together with \eqref{2.65} yields
\begin{align}\label{3.49}
\int_{0}^\delta\int_{0}^\pi\bigg((J-\theta_t)\varphi^0_t +J\varphi^0_{xx}\bigg)\ {\rm d}x{\rm d}t- \int_{0}^\delta J\varphi^0_x\bigg|_{x=0}^{x=\pi}\ {\rm d}t =0.
\end{align}
We take the difference of \eqref{3.49} and \eqref{3.53}, with the test function $\varphi^0$ vanishing on the boundaries, to obtain
\begin{align}\label{3.54}
\int_{0}^\delta\int_{0}^\pi[J-(u_x+\theta_t)](\varphi^0_t+ \varphi^0_{xx})\ {\rm d}x{\rm d}t- \int_{0}^\delta (J-u_x-\theta_t)\varphi^0_x\bigg|_{x=0}^{x=\pi}\ {\rm d}t=0.
\end{align}
Hence it is clear that $J=u_x+\theta_t$, for a.e. any $(x,t)\in[0,\pi]\times[0,\delta]$.

\subsection{Energy estimate}\label{S3.3}

In this subsection, we derive the energy estimate in Theorem \ref{thm}, which allows us to extend the local solution on $\Omega_\delta$ to $\Omega_T$.

Following \cite{CHL20}, we proceed by the inequality in \eqref{2.74c} and the fact $J=u_x+\theta_t$ in $L^2(\Omega_T)$ sense
\begin{align}\label{3.55}
\int_{0}^{\pi}(\theta_{t}^2+c^2(\theta)\theta_{x}^2)(x,t)\ {\rm d}x
\leq &  \int_{0}^{\pi} (\theta_{t}^2+c^2(\theta)\theta_{x}^2)(x,0)\ {\rm d}x + 2B(\theta(\pi,0)) \nonumber \\ & -2B(\theta(\pi,t)) -2\int_{0}^t\int_{0}^{\pi}\theta_{t}^2\ {\rm d}x{\rm d}t -2\int_{0}^t\int_{0}^{\pi}J\theta_{t}\ {\rm d}x{\rm d}t \nonumber \\
\leq  & \int_{0}^{\pi} (\theta_{t}^2+c^2(\theta)\theta_{x}^2)(x,0)\ {\rm d}x + 2B(\theta(\pi,0)) \nonumber \\
&-2B(\theta(\pi,t)) -2\int_{0}^t\int_{0}^{\pi}\theta_{t}^2\ {\rm d}x{\rm d}t -2\int_{0}^t\int_{0}^{\pi}J(J-u_x)\ {\rm d}x{\rm d}t.
\end{align}
Since $u_t=J_x$ holds in $L^2(\Omega_T)$ sense, one has by integrating by parts
\begin{align}\label{3.56}
\int_{0}^t\int_{0}^{\pi}Ju_x\ {\rm d}x{\rm d}t=&
-\int_{0}^t\int_{0}^{\pi}u_tu\ {\rm d}x{\rm d}t \nonumber \\
=&\fr{1}{2}\int_{0}^{\pi}u_{0}^2(x)\ {\rm d}x-\fr{1}{2}\int_{0}^{\pi}u^2(x,t)\ {\rm d}x.
\end{align}
Putting \eqref{3.56} into \eqref{3.55} yields
\begin{align}\label{3.57}
&\int_{0}^{\pi}(\theta_{t}^2+c^2(\theta)\theta_{x}^2)(x,t)\ {\rm d}x +2B(\theta(\pi,t))
\leq \int_{0}^{\pi} (\theta_{t}^2+c^2(\theta)\theta_{x}^2)(x,0)\ {\rm d}x +2B(\theta(\pi,0)) \nonumber \\ & -2\int_{0}^t\int_{0}^{\pi}\theta_{t}^2\ {\rm d}x{\rm d}t -2\int_{0}^t\int_{0}^{\pi}J^2\ {\rm d}x{\rm d}t + \int_{0}^{\pi}u_{0}^2(x)\ {\rm d}x-\int_{0}^{\pi}u^2(x,t)\ {\rm d}x,
\end{align}
which means that
\begin{align}\label{3.58}
&\fr{1}{2}\int_{0}^{\pi}(\theta_{t}^2+c^2(\theta)\theta_{x}^2 +u^2)(x,t)\ {\rm d}x +B(\theta(\pi,t)) \nonumber \\
\leq & \fr{1}{2}\int_{0}^{\pi} (\theta_{t}^2+c^2(\theta)\theta_{x}^2 +u^2)(x,0)\ {\rm d}x +B(\theta(\pi,0)) -\int_{0}^t\int_{0}^{\pi}(u_x+\theta_t)^2+\theta_{t}^2\ {\rm d}x{\rm d}t,
\end{align}
which is the desired inequality \eqref{1.14}.

\section{The problem with Neumann boundary conditions for heat equation}

In this section, we provide the proof on the initial-boundary value problem of \eqref{1.1} with the Neumann boundary conditions on $u$.

First, in Section 2, we solve a boundary value problem for $\theta$, so the result still holds for the case with the Neumann boundary conditions.

For the existence of $(u,\theta)$ by Schaulder fixed point theorem, to avoid repeat, we only give calculations different from the case with the Dirichlet boundary condition.

Consider system \eqref{1.1} with the initial condition \eqref{1.2} and the following the mixed boundary conditions
\begin{equation}\label{4.1}
\begin{aligned}
(u_x+\theta_t)(0,t)=(u_x+\theta_t)(\pi,t)&=0,\\
\theta(0,t)=\iota \theta(\pi,t)+\theta_x(\pi,t)&=0.
\end{aligned}
\end{equation}
Assume that the initial functions $u_0(x)$, $\theta_0(x)$ and $\theta_1(x)$ satisfy the corresponding compatibility conditions at $0$ and $\pi$.
Based on the results in Section \ref{S2}, for any given $T>0$ and $J\in C^\alpha$, we obtain the global solution $\theta(x,t)$ for the wave equation
$$
\theta_{tt}-c(\theta)(c(\theta)\theta_x)_x+\theta_t+J=0.
$$
Corresponding to the Neumann boundary conditions of $u$, the initial-boundary value problem of $J$ now is
\begin{align}\label{4.2}
\left\{
\begin{array}{l}
J_t-J_{xx}=c(\theta)(c(\theta)\theta_x)_x-\theta_t-J, \\
J(x,0)=J_0(x):=u_{0}'(x)+\theta_1(x)\in C^\alpha([0,\pi]), \\
J(0,t)=J(\pi,t)=0.
\end{array}
\right.
\end{align}
Thanks to the Green function $G(x,t;\xi,\tau)$ given in Section \ref{S3}, the solution $J$ of \eqref{4.2} can be expressed as
\begin{align}\label{4.3}
&J(x,t)
= \int_{0}^\pi G(x,t;\xi,0)J_0(\xi)\ {\rm d}\xi -\int_{0}^t\int_{0}^\pi G(x,t;\xi,\tau)(\theta_\tau+J)(\xi,\tau)  \ {\rm d}\xi{\rm d}\tau \nonumber \\
&\qquad -\int_{0}^t\int_{0}^\pi G(x,t;\xi,\tau)cc'\theta_{\xi}^2(\xi,\tau)\ {\rm d}\xi{\rm d}\tau -\int_{0}^t\int_{0}^\pi\pa_\xi G(x,t;\xi,\tau)c^2\theta_{\xi}(\xi,\tau)\ {\rm d}\xi{\rm d}\tau.
\end{align}
The relation \eqref{4.3} defines a map on $C^\alpha(\Omega_T)$
\begin{align}\label{4.4}
\widetilde{\mathcal{T}}:\ J(x,t)\rightarrow \widetilde{\mathcal{M}}(J)(x,t),
\end{align}
where
\begin{align}\label{4.5}
&\widetilde{\mathcal{M}}(J)(x,t)
= \int_{0}^\pi G(x,t;\xi,0)J_0(\xi)\ {\rm d}\xi -\int_{0}^t\int_{0}^\pi G(x,t;\xi,\tau)(\theta_\tau+J)(\xi,\tau)  \ {\rm d}\xi{\rm d}\tau \nonumber \\
&\qquad -\int_{0}^t\int_{0}^\pi G(x,t;\xi,\tau)cc'\theta_{\xi}^2(\xi,\tau)\ {\rm d}\xi{\rm d}\tau -\int_{0}^t\int_{0}^\pi\pa_\xi G(x,t;\xi,\tau)c^2\theta_{\xi}(\xi,\tau)\ {\rm d}\xi{\rm d}\tau,
\end{align}
which is a weak solution of the following equation
\begin{align}\label{4.6}
\widetilde{\mathcal{M}}_t-\widetilde{\mathcal{M}}_{xx}=c(\theta)(c(\theta)\theta_x)_x-\theta_t-J.
\end{align}
Through the parallel procedure as in Subsection \ref{S3.2}, we can verify that the map $\widetilde{\mathcal{T}}$ exists a fixed point in the set
\begin{align}\label{4.7}
\widetilde{\mathcal{K}}=\bigg\{&J(x,t)\big|\ \|J(x,t)-\widetilde{J}^0(x,t)\|_{C^\alpha(\Omega_\delta)}\leq K,  J(x,0)=J_0(x),\ J(0,t)=J(\pi,t)=0\bigg\},
\end{align}
for some constants $\delta>0$ and $K>0$, where
$$
\widetilde{J}^0(x,t)=\int_{0}^{\pi}G(x,t;\xi,0)J_0(\xi)\ {\rm d}\xi.
$$
Denote this fixed point still by $J\in \widetilde{\mathcal{K}}$. Then there has
\begin{align}\label{4.8}
\widetilde{\mathcal{M}}(J)(x,t)=J(x,t).
\end{align}

We next show $J=u_x+\theta_t$. It follows by \eqref{4.6} and \eqref{4.8} that the weak solution $J$ of \eqref{4.2} satisfies
\begin{align}\label{4.9}
\int_{0}^\delta\int_{0}^\pi\bigg(J(\varphi^{0}_t+\varphi^{0}_{xx}) -[(c\varphi)_xc\theta_x+(\theta_t+J)\varphi^{0}]\bigg)\ {\rm d}x{\rm d}t +\int_{0}^\delta(J_x+c^2\theta_x)\varphi^{0}\bigg|_{x=0}^{x=\pi}\ {\rm d}t=0,
\end{align}
for any test function $\varphi^0(x,t)\in C^\infty_c([0,\pi]\times[0,\delta])$.
Combining \eqref{2.65} and \eqref{4.9}, one obtains
\begin{align}\label{4.10}
\int_{0}^\delta\int_{0}^\pi\bigg((J-\theta_t)\varphi^{0}_t +J\varphi^{0}_{xx}\bigg)\ {\rm d}x{\rm d}t +\int_{0}^\delta J_x\varphi^{0}\bigg|_{x=0}^{x=\pi}\ {\rm d}t=0.
\end{align}
Furthermore, for the variable $u$, we consider the initial-boundary value problem
\begin{align}\label{4.11}
\left\{
\begin{array}{l}
u_t-u_{xx}=\theta_{tx},\\
u(x,0)=u_0(x),\\
u_x(0,t)=0,\ \ u_x(\pi,t)=-\theta_t(\pi,t).
\end{array}
\right.
\end{align}
Let
$$
\tilde{u}(x,t)=u(x,t)+\fr{1}{\pi}\int_{0}^x y\theta_t(y, t)\ {\rm d}y.
$$
Then the initial-boundary value problem for $\tilde{u}$ is
\begin{align}\label{4.11a}
\left\{
\begin{array}{l}
\dps \tilde{u}_t-\tilde{u}_{xx}=(1-\fr{x}{\pi})\theta_{tx} -\fr{1}{\pi}\theta_t +\fr{1}{\pi}xc^2\theta_x -\fr{1}{\pi}\int_{0}^x(c+yc'\theta_y)c\theta_y\ {\rm d}y -\fr{1}{\pi}\int_{0}^x y(\theta_t+J)\ {\rm d}y,\\[8pt]
\dps \tilde{u}(x,0)=u_0(x)+\fr{1}{\pi}\int_{0}^x y\theta_1(y)\ {\rm d}y,\quad x\in[0,\pi], \\[8pt]
\dps \tilde{u}_x(0,t)=\tilde{u}_x(\pi,t)=0.
\end{array}
\right.
\end{align}
Here we have used the wave equation for $\theta(x,t)$. According to the Neumann function $N(x,t;\xi,\tau)$ given in Section \ref{S3}, the weak solution of \eqref{4.11} can be expressed as
\begin{align}\label{4.12}
&u(x,t)=-\fr{1}{\pi}\int_{0}^x y\theta_t(y, t)\ {\rm d}y +\int_{0}^\pi N(x,t;\xi,0)\bigg(u_0(\xi) +\fr{1}{\pi}\int_{0}^\xi y\theta_1(y)\ {\rm d}y\bigg)\ {\rm d}\xi \nonumber \\
&\ \ +\int_{0}^t\int_{0}^\pi N(x,t;\xi,\tau)\bigg(\fr{1}{\pi}\xi c^2\theta_\xi -\fr{1}{\pi}\int_{0}^\xi(c+yc'\theta_y)c\theta_y\ {\rm d}y -\fr{1}{\pi}\int_{0}^\xi y(\theta_\tau+J)\ {\rm d}y\bigg)(\xi,\tau)\ {\rm d}\xi{\rm d}\tau \nonumber \\
&\ \ -\int_{0}^t\int_{0}^\pi\pa_\xi N(x,t;\xi,\tau)\left(1-\fr{\xi}{\pi}\right)\theta_\tau(\xi,\tau)\ {\rm d}\xi{\rm d}\tau,
\end{align}
which satisfies
\begin{align}\label{4.13}
\int_{0}^\delta\int_{0}^\pi\bigg(u\psi^{0}_t-(u_x+\theta_t)\psi^{0}_x\bigg)\ {\rm d}x{\rm d}t=0,
\end{align}
for any test functions $\psi^{0}\in C^\infty_c([0,\pi]\times[0,\delta])$. Letting $\psi^{0}=\varphi^{0}_x$ in \eqref{4.13}, we get
\begin{align}\label{4.14}
\int_{0}^\delta\int_{0}^\pi\bigg(u_x\varphi^{0}_t+(u_x+\theta_t)\varphi^{0}_{xx}\bigg)\ {\rm d}x{\rm d}t -\int_{0}^\delta u\varphi^{0}_t\bigg|_{x=0}^{x=\pi}\ {\rm d}t=0.
\end{align}
By taking the difference of \eqref{4.10} and \eqref{4.14} yields
\begin{align}
\int_{0}^\delta\int_{0}^\pi[J-(u_x+\theta_t)](\varphi^{0}_t+ \varphi^{0}_{xx})\ {\rm d}x{\rm d}t + \int_{0}^\delta (u\varphi^{0}_t+J_x\varphi^{0})\bigg|_{x=0}^{x=\pi}\ {\rm d}t =0,
\end{align}
which implies by the disappearance of the test function $\varphi^{0}$ on the boundaries that $J=u_x+\theta_t$, for a.e. any $(x,t)\in[0,\pi]\times[0,\delta]$.

Finally, the energy estimate in \eqref{3.58} still holds for the current case, so
we can extend the local solution on $\Omega_\delta$ to $\Omega_T$.

\section*{Acknowledgments}
 The first author is partially
supported by NSF with grant DMS-2008504. The second author is partially supported by NSF of China with grants 12071106 and 12171130.


\begin{thebibliography}{a}
\bibitem{BC2015} A. Bressan and G. Chen,
Lipschitz metric for a class of nonlinear wave equations,
{\em Arch. Ration. Mech. Anal.}, {\bf 226} (2017), no. 3, 1303--1343.

\bibitem{BC} A. Bressan and G. Chen,
Generic regularity of conservative solutions to a nonlinear wave equation,
{\em Ann. I. H. Poincar\'{e}--AN}, {\bf  34} (2017),  no. 2, 335--354.

\bibitem{BCZ2} A. Bressan, G. Chen, and Q. T. Zhang, Unique conservative solutions to a variational wave equation, {\em Arch. Ration. Mech. Anal.}, {\bf 217} (2015), no. 3, 1069--1101.

\bibitem{BH} A. Bressan and T. Huang,
Representation of dissipative solutions to a nonlinear variational wave equation,
{\em Commun. Math. Sci.}, {\bf 14} (2016), no. 1, 31--53.


\bibitem{Bressan-Zheng} A. Bressan and Y. X. Zheng, Conservative solutions to a nonlinear variational wave equation, {\em Commun. Math. Phys.}, {\bf 266} (2006), no. 2, 471--497.

\bibitem{CCD} H. Cai, G. Chen, and Y. Du,
Uniqueness and regularity of conservative solution to a wave system modeling nematic liquid crystal,
{\em J. Math. Pures Appl.}, {\bf 9} (2018), no. 117, 185--220.


\bibitem{CCS} H. Cai, G. Chen, and Y. Du,
Lipschitz optimal transport metric for a wave system modeling nematic liquid crystals, {\em submitted}, available at arXiv: 2304.11535.

\bibitem{CHL20} G. Chen, T. Huang, and W. S. Liu,
Poiseuille flow of nematic liquid crystals via the full Ericksen-Leslie model, {\em  Arch. Ration. Mech. Anal.}, {\bf 236} (2020), no. 2, 839--891.

\bibitem{CHX} G. Chen, T. Huang, and X. Xu
Singularity formation for full Ericksen-Leslie system of nematic liquid crystal flows in dimension two,
 {\em submitted}, available at arXiv: 2305.03904.

\bibitem{CLS} G. Chen, W. S. Liu, and M. Sofiani, The Poiseuille flow of the full Ericksen-Leslie model for nematic liquid crystals: The general Case, {\em submitted}, available at arXiv: 2302.08616.

\bibitem{CZZ12} G. Chen, P. Zhang, and Y. X. Zheng,
Energy Conservative Solutions to a Nonlinear Wave
System of Nematic Liquid Crystals,
{\em Comm. Pure Appl. Anal.}, {\bf 12} (2013), no. 3, 1445--1468.


\bibitem{CJL21}
F. Cheng, N. Jiang, and Y. L. Luo,
On dissipative solutions to a simplified hyperbolic Ericksen-Leslie system of liquid crystals, {\em Commun. Math. Sci.}, {\bf 19} (2021), no. 1, 175--192.

\bibitem{H2} W. H. Duan, Y. B. Hu, and G. D. Wang, Singularity for a multidimensional variational wave equation arising from nematic liquid crystals, {\em J. Math. Anal. Appl.}, {\bf 487} (2020), no. 2, 124026, 13 pp.

\bibitem{Eric61} J. L. Ericksen, Conservation laws for liquid crystals, {\em Trans. Soc. Rheology}, {\bf 5} (1961), no. 1, 23--34.

\bibitem{Eric87}  J. L. Ericksen, Continuum theory of nematic liquid crystals, {\em Res. Mechanica}, {\bf 21 } (1987), 381--392.

\bibitem{Eric90} J. L. Ericksen, Liquid crystals with variable degree of orientation, {\em Arch. Ration. Mech. Anal.}, {\bf 113} (1991), no. 2, 97--120.

\bibitem{GHZ} R. T. Glassey, J. K. Hunter, and Y. X. Zheng,
Singularities in a nonlinear variational wave equation,
{\em J. Differential Equations}, {\bf 129} (1996), no. 1, 49--78.

\bibitem{HardtK87} R. Hardt and D. Kinderlehrer,
Mathematical questions of liquid crystal theory,
In Theory and applications of liquid crystals (Minneapolis, Minn., 1985),
{\em IMA Vol. Math. Appl.} {\bf Vol 5}, pp. 151--184, Springer, New York, 1987.


\bibitem{HR} H.~Holden and X.~Raynaud,
Global semigroup of conservative solutions of the nonlinear variational wave equation,
{\em Arch. Ration. Mech. Anal.}, {\bf 201} (2011), no. 3, 871--964.


\bibitem{hongxin12} M.~C. Hong and Z.~P. Xin,
\newblock Global existence of solutions of the liquid crystal flow for the
  {O}seen-{F}rank model in {$\mathbb R^2$},
\newblock {\em Adv. Math.}, {\bf 231} (2012), no. 3-4, 1364--1400.

\bibitem{H} Y. B. Hu, Conservative solutions to a one-dimensional nonlinear variational wave equation, {\em J. Differential Equations}, {\bf 259} (2015), no. 1, 172--200.

\bibitem{huanglinwang14} J. R. Huang, F. H. Lin, and C. Y. Wang,
Regularity and existence of global solutions to the Ericksen-Leslie system in $\mathbb R^2$,
{\em Commun. Math. Phys.}, {\bf 331} (2014), no. 2, 805--850.


\bibitem{HJLZ21}
J. X. Huang, N. Jiang, Y. L. Luo, and L. F. Zhao,
Small data global regularity for the 3-D Ericksen-Leslie hyperbolic liquid crystal model without kinematic transport, {\em SIAM J. Math. Anal.}, {\bf 53} (2021), no. 1, 530--573.

\bibitem{JL19} N. Jiang and Y. L. Luo, On well-posedness of Ericksen-Leslie's hyperbolic incompressible liquid crystal model, {\em SIAM J. Math. Anal.}, {\bf 51} (2019), no. 1, 403--434.

\bibitem{JL22} N. Jiang and Y. L. Luo,
The zero inertia limit from hyperbolic to parabolic Ericksen-Leslie system of liquid crystal flow, {\em J. Funct. Anal.}, {\bf 282} (2022), no. 1, 109280, 62 pp.

\bibitem{JLT19}
N. Jiang, Y. L. Luo, and S. J. Tang,
Zero inertia density limit for the hyperbolic system of Ericksen-Leslie's liquid crystal flow with a given velocity, {\em  Nonlinear Anal. Real World Appl.}, {\bf 45} (2019), 590--608.


\bibitem{Les79} F. Leslie, Theory of flow phenomena in liquid crystals, {\em Advances in Liquid Crystals}, 4, Elsevier, New York, 1979, pp. 1-81.

 \bibitem{LTX16} J. K. Li, E. Titi, and Z. P. Xin,
 On the uniqueness of weak solutions to the Ericksen-Leslie liquid
crystal model in $\mathbb R^2$,
{\em Math. Models Methods Appl. Sci.}, {\bf  26} (2016), no. 4, 803--822.


\bibitem{Li-Yu-Shen1} T. T. Li, W. C. Yu, and W. X. Shen, Second initial-boundary value problems for quasi-linear hyperbolic-parabolic coupled systems (in chinese), {\em Chinese Ann. of Math.}, {\bf  2} (1981), 65--90.

\bibitem{Li-Yu-Shen2} T. T. Li, W. C. Yu, and W. X. Shen, First initial-boundary value problems for quasilinear hyperbolic-parabolic coupled systems, {\em Chinese Ann. of Math.}, {\bf 5B} (1984), 77--90.

\bibitem{lin89} F. H. Lin,
Nonlinear theory of defects in nematic liquid crystals;  phase transition and phenomena,
{\em Comm. Pure Appl. Math.}, {\bf 42} (1989), no. 6, 789--814.

 \bibitem{linlius01} F. H. Lin and C. Liu,
 Static and dynamic theories of liquid crystals,
 {\em J. Partial Differential Equations}, {\bf 14} (2001), no. 4, 289--330.

\bibitem{linwangs14} F. H. Lin and C. Y. Wang,
 Recent developments of analysis for hydrodynamic flow of nematic liquid crystals,
{\em Philos. Trans. R. Soc. Lond. Ser. A Math. Phys. Eng. Sci.}, {\bf 372} (2014),
no. 2029, 20130361, 18 pp.

\bibitem{linwang16} F. H. Lin and C. Y. Wang,
Global existence of weak solutions of the nematic liquid crystal
flow in dimension three, {\em Comm. Pure Appl. Math.},  {\bf 69} (2016), no. 8, 1532--1571.

\bibitem{Ste} I. W. Stewart,  The Static and Dynamic Continuum Theory of Liquid Crystals: A Mathematical Introduction, Crc Press, 2019.

\bibitem{wangwang14} M. Wang and W. D. Wang,
Global existence of weak solution for the 2-D Ericksen-Leslie system,
{\em Calc. Var. Partial Differential Equations}, {\bf 51} (2014), no. 3-4, 915--962.

\bibitem{WZZ13} W. Wang, P. W. Zhang, and Z. F. Zhang,
Well-posedness of the Ericksen-Leslie system,
{\em Arch. Ration. Mech. Anal.}, {\bf 210} (2013), no. 3, 837-855.

\bibitem{ZZ03} P. Zhang and Y. X. Zheng,
Weak solutions to a nonlinear variational wave equation,
{\em Arch. Ration. Mech. Anal.}, {\bf 166} (2003), no. 4, 303--319.

\bibitem{ZZ10} P. Zhang and Y. X. Zheng, Conservative solutions to
a  system of variational wave equations of nematic liquid crystals,
{\em Arch. Ration. Mech. Anal.}, {\bf 195} (2010), no. 3, 701--727.

\bibitem{ZZ11} P. Zhang and Y. X. Zheng,
Energy conservative solutions to a one-dimensional full variational wave system,
{\em Comm. Pure Appl. Math.}, {\bf 55} (2012), no. 5, 582--632.

\end{thebibliography}
\end{document}